\documentclass[letterpaper,10pt,reqno]{amsart}

\usepackage{amssymb}
\usepackage{latexsym}
\usepackage{amsbsy}
\usepackage{amsfonts}
\usepackage{comment}

\usepackage{hyperref}
\usepackage{graphicx}
\usepackage{enumerate}
\usepackage{enumitem}
\usepackage{color}

\usepackage{stackengine}
\usepackage{multirow}
\usepackage{caption}
\usepackage{adjustbox}

\usepackage{tikz}

\def\marginpar#1{\ignorespaces}

\DeclareMathOperator\argmax{argmax}

\DeclareMathOperator\var{Var}
\newtheorem{theorem}{Theorem}[section]
\newtheorem{lemma}[theorem]{Lemma}
\newtheorem{proposition}[theorem]{Proposition}
\newtheorem{corollary}[theorem]{Corollary}
\newtheorem{definition}[theorem]{Definition}

\numberwithin{equation}{section}

\newcommand{\beq}{\begin{equation}}
\newcommand{\eeq}{\end{equation}}

\newcommand{\bal}{\begin{align}}
\newcommand{\eal}{\end{align}}
\newcommand{\bals}{\begin{align*}}
\newcommand{\eals}{\end{align*}}

\makeatother
\begin{document}
\title[Fine-tuning via stochastic control]{Fine-tuning of diffusion models via stochastic control: entropy regularization and beyond}

\author[Wenpin Tang and Fuzhong Zhou]{{Wenpin} Tang and Fuzhong Zhou}
\address{Department of Industrial Engineering and Operations Research, Columbia University. 
} \email{wt2319@columbia.edu, fz2329@columbia.edu}

\date{\today} 
\begin{abstract}
This paper aims to develop and provide a rigorous treatment 
to entropy regularized fine-tuning in the context of continuous-time diffusion models,
which was proposed by Uehara et al. (arXiv:2402.15194, 2024).
The idea is to use stochastic control for sample generation,
where the entropy regularizer is introduced to mitigate reward collapse.
We also show how the analysis can be extended to fine-tuning 
with a general $f$-divergence regularizer.
Numerical experiments on large-scale text-to-image models -- Stable Diffusion v1.5
are conducted to validate our approach.
\end{abstract}

\maketitle

\textit{Key words}: Constrained optimization, diffusion models, entropy-regularization, $f$-divergence, fine-tuning, stochastic control, stochastic differential equations.

\textit{AMS 2020 Mathematics Subject Classification: 60J60, 65C30, 93E20.}

\section{Introduction}

\quad Diffusion models \cite{Ho20, Sohl15, Song20}
are a promising generative approach to produce high-quality samples,
which have been observed to outperform adversarial generative nets 
in image and audio synthesis \cite{Dh21, Kong21},
and underpin recent success in text-to-image creators such as 
DALL$\cdot$E2 \cite{Ramesh22} and Stable Diffusion \cite{Rombach22},
and the text-to-video generator Sora \cite{Sora}.
Though diffusion models can capture intricate and high-dimensional
data distributions, 
they may suffer from sources of bias or fairness concerns \cite{LA23},
and 
the training process (especially for the aforementioned large models)
requires considerable time and effort. 

\quad There is growing interest in improving diffusion models
in terms of generated sample quality, as well as controllability.
One straightforward approach is to
fine-tune the sampler customized for a specific task 
using the pretrained (diffusion) model
as a base model. 
For instance, in image/video generation, 
we aim to fine-tune diffusion models to 
enhance the aesthetic quality
and prevent distorted contents.
With the emergence of human-interactive platforms such as ChatGPT,
there is substantial demand to align
generative models 
with user/human preference or feedback.
Recent works \cite{Black23, Fan23, FW23, ZZT24} 
proposed to fine-tune diffusion models by reinforcement learning (RL),
and \cite{WD23} by direct preference optimization.
In these works,  
the reward functions are learned statistical models,
e.g., an aesthetic reward in image generation is a ranking model 
fit to the true aesthetic preferences of human raters.
Refer to \cite{CMFW24, Ueh24, Win25} for reviews on 
preference tuning of diffusion models.

\quad The above methods allow to fine-tune the diffusion model
to generate samples with high nominal rewards. 
However, they may lead to {\em reward collapse} or {\em hacking} 
\cite{Amo16, SC23},
a phenomenon referring to overfitting the reward 
(that is trained by e.g., a limited number of human ratings.)
In other words, the diffusion model is fine-tuned with respect to 
some human-rated score
which may fail to generalize.
Moreover, exploiting the reward exclusively
also harms {\em diversity}, 
a criterion that is at the heart of generative modeling.

\quad In order to mitigate reward collapse/hacking
and enhance diversity,
\cite{UZ24} proposed to add
an entropy regularizer with respect to the pretrained model
in the loss objective. 
This yields the {\em entropy-regularized fine-tuning},
which is an exponential tilting of that generated by
the pretrained model
and
can be viewed as a ``soft" {\em diffusion guidance}
\cite{Dh21, HS21, TX25}.
A stochastic control approach was developed to
emulate this novel distribution. 
In comparison with the standard control framework where
the initial distribution is fixed,
both the control and the initial distribution are decision variables.
The resulting problem is to ``decouple" these two variables by first solving a standard control problem,
followed by finding the optimal initial distribution.
To the best of our knowledge, 
this is the first time that stochastic control is used for
diffusion sample generation.
In the concurrent work \cite{DC25}, stochastic control is also used for fine-tuning
diffusion models, but without learning the initial distribution
(instead, it requires cautious model selection.)
The idea of adding an entropy regularizer
also appeared in \cite{ZZT24}
for fine-tuning the diffusion model via continuous RL, 
and in \cite{Kong24} for constraint optimization via the diffusion model.
In a different context, (entropy regularized) stochastic control was proposed 
to solve
non-convex problems \cite{GXZ22, TZZ22},
as an alternative to the simulated annealing algorithm.

\quad The purpose of this paper is two-fold.
First, we provide a rigorous treatment to the theory of entropy-regularized fine-tuning 
proposed in \cite{UZ24}. 
The paper is mostly self-contained. 
Second, we extend to the problem of fine-tuning 
regularized by general $f$-divergence,
and show how the analysis carries over.
The remainder of the paper is organized as follows. 
In Section \ref{sc2},
we recall background of diffusion models.
The entropy-regularized fine-tuning is studied in Section \ref{sc3},
and the extension to fine-tuning regularized by $f$-divergence
is given in Section \ref{sc4}.
We conduct numerical experiments in Section \ref{scnu}.
Concluding remarks are summarized in Section \ref{sc5}.

\medskip
{\em Notations}: Below we collect a few notations that will be used throughout.
\begin{itemize}[itemsep = 3 pt]
\item
$\mathbb{R}_{+}$ denotes the set of nonnegative real numbers.
\item
For $x, y$ vectors, denote by $x \cdot y$ the inner product between $x$ and $y$,
and $|x|$ is the Euclidean ($L^2$) norm of $x$.
\item
For $p(\cdot)$ a probability distribution on $\mathbb{R}^d$, 
we assume that it has a density $p(x)$, and write $p(dx) = p(x) dx$. 
\item
The notation $X \sim p(\cdot)$ means that the random variable is distributed according to $p(\cdot)$.
We write $\mathbb{E}_{p(\cdot)}(X)$ for the expectation of $X \sim p(\cdot)$.
\item
For $p(\cdot)$ and $q(\cdot)$ two probability distributions, 
$D_{TV}(p(\cdot), q(\cdot)):= \sup_A |p(A) - q(A)|$ is the total variation distance between $p(\cdot)$ and $q(\cdot)$,
and 
$D_{KL}(p(\cdot), q(\cdot)):= \int \log \frac{dp}{dq} dp$ is the KullbackLeilber (KL) divergence between
$p(\cdot)$ and $q(\cdot)$.
\end{itemize}

\section{Diffusion models}
\label{sc2}

\quad In this section, we review diffusion models
that will be used as the pretrained models in fine-tuning. 
We follow closely the presentation in \cite{TZ24tut}.

\quad The goal of diffusion models is to generate new samples (e.g., images, video, text) that 
resemble the target data, while maintain a certain level of diversity. 
Diffusion modeling relies on a forward-backward procedure:
\begin{itemize}[itemsep = 3 pt]
\item
{\em Forward deconstruction}:
start from the target distribution $X_0 \sim p_{\tiny \mbox{data}}(\cdot)$,
the model gradually adds noise to transform the signal into noise
$X_0 \to X_1 \to \cdots \to X_n \sim p_{\tiny \mbox{noise}}(\cdot)$.
\item
{\em Backward reconstruction}: 
start with the noise $X_n \sim p_{\tiny \mbox{noise}}(\cdot)$, and reverse the forward
process to recover the signal from noise $X_n \to X_{n-1} \to \cdots \cdot X_0 \sim p_{\tiny \mbox{data}}(\cdot)$.
\end{itemize}
Here we consider the diffusion model in continuous time,
governed by the stochastic differential equation (SDE):
\begin{equation}
\label{eq:SDE}
dX_t = b(t,X_t) dt + \sigma(t) dW_t, \quad X_0 \sim p_{\tiny \mbox{data}}(\cdot),
\end{equation}
where 
$(W_t, \, t \ge 0)$ is $d$-dimensional Brownian motion,
and
$b: \mathbb{R}_+ \times \mathbb{R}^d \to \mathbb{R}^d$
and $\sigma: \mathbb{R}_+ \to \mathbb{R}_+$ are model parameters. 
Some conditions on $b(\cdot, \cdot)$, $\sigma(\cdot)$
are required so that 
the SDE \eqref{eq:SDE} is well-defined 
(see \cite[Chapter 5]{KS91}, \cite{SV79}).
In the sequel,
we assume for simplicity that
the distribution of $X_t$ has a suitably smooth density,
and denote by $p(t,x):=\mathbb{P}(X_t \in dx)/dx$.

\quad \quad The key to the success of continuous-time diffusion models is that
their time reversal has a rather tractable form, 
making it more convenient for methodological development.
It is known \cite{Ander82, HP86} that
the distribution of the time reversal process $\overline{X}_t = X_{T-t}$ is governed by the SDE:
\begin{equation*}
d\overline{X}_t = (-b(T-t, \overline{X}_t) + \sigma^2(T-t) \nabla \log p(T-t, \overline{X}_t)) dt + \sigma(T-t) dB_t, \quad \overline{X}_0 \sim p(T, \cdot),
\end{equation*}
where $(B_t, \, t \ge 0)$ is a copy of $d$-dimensional Brownian motion.
Thus, the process  recovers 
$\overline{X}_T \sim p_{\tiny \mbox{data}}(\cdot)$ at time $T$.

\quad There are, however, two twists in diffusion modeling.  
First, diffusion models aim to generate the target distribution from {\em noise},
and the noise should not depend
on the target distribution.
So instead of setting $\overline{X}_0 \sim p(T, \cdot)$,
the backward process is initiated with some noise
$p_{\tiny \mbox{noise}}(\cdot)$ as a proxy of 
$p(T, \cdot)$:
\begin{equation}
\label{eq:reversal}
d\overline{X}_t = \left(-b(T-t, \overline{X}_t) + \sigma^2(T-t) \nabla \log p(T-t, \overline{X}_t)\right) dt + \sigma(T-t) dB_t, \quad \overline{X}_0 \sim p_{\tiny \mbox{noise}}(\cdot).
\end{equation}
The diffusion model is specified by the pair $(b(\cdot, \cdot), \sigma(\cdot))$,
and existing examples include 
Ornstein-Ulenback processes \cite{DV21},
variance exploding (VE) SDEs, 
variance preserving (VP) SDEs \cite{Song20},
and contractive diffusion models \cite{TZ24}.
The choice of $p_{\tiny \mbox{noise}}(\cdot)$
depends on each specific model,
and varies case by case. 

\quad Second, 
all but the term $\nabla \log p(T-t, \overline{X}_t)$ 
in \eqref{eq:reversal} are available.
So in order to implement the SDE \eqref{eq:reversal}, 
we need to compute $\nabla \log p(t,x)$,
known as {\em Stein's score function}.
Recently developed {\em score matching} techniques 
allow to estimate the score function
via function approximations. 
More precisely, 
we use a family of functions $\{s_\theta(t,x)\}_\theta$ (e.g., neural nets)
to approximate the score $\nabla \log p (t,x)$. 
The most widely used approach is {\em denoising score matching} \cite{Vi11}:
\begin{equation*}
\min_\theta \mathbb{E}_{t \sim \tiny \mbox{Unif}\, [0,T], \, X_0\sim p_{data}(\cdot)} \bigg[
\mathbb{E}_{p(t, \cdot | X_0)} \Big|s_{\theta}(t,X_t)- \nabla \log p(t,X_t | X_0)\Big|^2 \bigg].
\end{equation*}
See \cite[Section 4]{TZ24tut} for a review.

\quad Now with the (true) score function $\nabla \log p (t,x)$
being replaced with the score matching function $s_\theta(t,x)$, 
the backward process is set to:
\begin{equation}
\label{eq:diffY}
dY_t = \underbrace{\left(-b(T-t, Y_t) + \sigma^2(T-t) s_\theta(T-t, Y_t)\right)}_{\overline{b}(t, Y_t)}  dt + \sigma(T-t) dB_t, \quad Y_0 \sim p_{\tiny \mbox{noise}}(\cdot).
\end{equation}
Here we use the symbol $Y$ to represent the process using the score matching $s_\theta(\cdot, \cdot)$,
which is distinguished from the process $\overline{X}$ 
(with the true score $\nabla \log p (t,x)$) defined by \eqref{eq:reversal}.
The SDE \eqref{eq:diffY} provides the generic form of the diffusion sampling,
which serves as a pretrained model for further goal-directed fine-tuning. 
Denote by $Q_{[0,T]}(\cdot)$ the probability distribution (on the path space) 
of the process $(Y_t, \, 0 \le t \le T)$,
and $Q_t(\cdot)$ its marginal distribution at time $t$. 
Write 
\begin{equation}
\label{eq:pretrain}
Y_T \sim  Q_T(\cdot)=:p_{\tiny \mbox{pre}}(\cdot).
\end{equation}

\section{Entropy-regularized fine-tuning and stochastic control}
\label{sc3}

\quad In this section, we consider the problem of entropy-regularized fine-tuning 
that aims to mitigate the reward collapse in the context of diffusion models.
We provide a rigorous treatment for the closely related
stochastic control problem, following the idea of \cite{UZ24}.
As we will see,
the arguments used in this section can be generalized
to study the problem 
of fine-tuning with an $f$-divergence regularizer.

\subsection{Entropy-regularized fine-tuning}
\label{sc31}

Let's explain entropy-regularized fine-tuning using a pretrained diffusion model \eqref{eq:diffY}.
The idea of fine-tuning is to calibrate a (possibly powerful) pretained model 
to some downstream task with a specific goal,
e.g. the aesthetic quality of generated images.
Such a goal is captured by a reward function 
$r: \mathbb{R}^d \to \mathbb{R}_+$, 
and an obvious way is to
maximize $\mathbb{E}_{\widetilde{Q}_{[0,T]}(\cdot)}[r(Y_T)]$
over a class of fine-tuned diffusion models $\widetilde{Q}_{[0,T]}(\cdot)$.
As mentioned in the introduction, 
this approach may generate samples
that overfit the reward,
the phenomenon known as 
the reward collapse/hacking.
In order to avoid the reward collapse
and utilize the pretrained model,
it is natural to consider the optimization problem:
\begin{equation}
\label{eq:entFT}
p_{\tiny \mbox{ftune}}(\cdot):=
\argmax_p \mathbb{E}_{p(\cdot)}[r(Y)] - \alpha D_{KL}(p(\cdot), p_{\tiny \mbox{pre}}(\cdot)),
\end{equation}
where the maximization is over all probability distribution on $\mathbb{R}^d$,
and 
the hyperparameter $\alpha > 0$ controls the level of exploration
relative to the pretrained model. 
The problem \eqref{eq:entFT} is referred to as entropy-regularized fine-tuning
of the pretrained model \eqref{eq:diffY}.

\quad It is easily seen that the problem \eqref{eq:entFT} has a closed-form solution.
\begin{lemma}
\label{lem:DVvar}
Assume that $\int \exp\left(\frac{r(y)}{\alpha} \right) p_{\tiny \mbox{pre}}(y) dy < \infty$.
We have:
\begin{equation}
\label{eq:cfsol}
p_{\tiny \mbox{ftune}}(y) = \frac{1}{C} \exp\left(\frac{r(y)}{\alpha} \right) p_{\tiny \mbox{pre}}(y),
\end{equation}
where $C:=\int \exp\left(\frac{r(y)}{\alpha} \right) p_{\tiny \mbox{pre}}(y) dy$ is the normalizing constant.
\end{lemma}
\begin{proof}
This follows from the duality formula for the variational problem in the theory of large deviations
\cite{DV83}.
\end{proof}

\quad At a high level, 
the target data with distribution $p_{\tiny \mbox{data}}(\cdot)$ is used to build 
the pretrained model that generates $p_{\tiny \mbox{pre}}(\cdot) \approx p_{\tiny \mbox{data}}(\cdot)$,
and the pretrained model is fine-tuned via some reward function to produce $p_{\tiny \mbox{ftune}}(\cdot)$.
Note that this differs from the setting of rejection sampling 
in such a way that we can only sample $p_{\tiny \mbox{pre}}(\cdot)$ 
and evaluate (usually in blackbox) via $r(\cdot)$ 
without knowing their exact forms. 

\quad As we will see in the next subsection, 
the main idea is to apply a (stochastic) control to 
the diffusion sampler \eqref{eq:diffY}
in order to generate $p_{\tiny \mbox{ftune}}(\cdot)$ in just one pass.
The next result shows how the fine-tuned distribution \eqref{eq:cfsol}
differs from $p_{\tiny \mbox{data}}(\cdot)$.
\begin{proposition}
\label{prop:TVd}
We have:
\begin{equation}
\label{eq:TVd}
D_{TV}(p_{\tiny \mbox{ftune}}(\cdot), p_{\tiny \mbox{data}}(\cdot))
\le D_{TV}(p_{\tiny \mbox{pre}}(\cdot), p_{\tiny \mbox{data}}(\cdot)) + \frac{1}{2} \sqrt{\mathbb{E}_{p_{\tiny \mbox{pre}}(\cdot)} \left[e^{\frac{2 r(Y)}{\alpha}} \right]}.
\end{equation}
\end{proposition}

\quad Before proving the proposition, let's make some remarks on the two terms
on the right side of \eqref{eq:TVd}:
\begin{itemize}[itemsep = 3 pt]
\item
The term $D_{TV}(p_{\tiny \mbox{pre}}(\cdot), p_{\tiny \mbox{data}}(\cdot))$ quantifies 
the performance of the pretrained model, 
i.e. how well (or how bad) the diffusion model recovers the target data distribution $p_{\tiny \mbox{data}}(\cdot)$.
Assuming that $\sigma(\cdot)$ is bounded away from zero,
and under the blackbox assumption that 
$\mathbb{E}_{p(t, \cdot)}|\nabla \log p(t, X) - s_\theta(t, X)|^2 \le \varepsilon^2$ for all $t$,
\cite{Chen23, TZ24tut} shows that
\begin{equation}
D_{TV}(p_{\tiny \mbox{pre}}(\cdot), p_{\tiny \mbox{data}}(\cdot))
\le D_{TV}(p_{\tiny \mbox{noise}}(\cdot), p(T, \cdot)) + \varepsilon \sqrt{T}/2.
\end{equation}
The term $D_{TV}(p_{\tiny \mbox{noise}}(\cdot), p(T, \cdot))$ can further be quantified 
under each specific model. 
As an example, for VP (with $b(t,x) = -\frac{1}{2} \beta(t) x$ and $\sigma(t) = \sqrt{\beta(t)}$),
we have $D_{TV}(p_{\tiny \mbox{noise}}(\cdot), p(T, \cdot)) \le e^{-\frac{1}{2} \int_0^T \beta(t) dt} \sqrt{\mathbb{E}_{p_{\tiny \mbox{data}}(\cdot)}|X|^2/2}$.
So the bound \eqref{eq:TVd} specializes to:
\begin{equation}
\label{eq:VP1}
D_{TV}(p_{\tiny \mbox{ftune}}(\cdot), p_{\tiny \mbox{data}}(\cdot)) \le
e^{-\frac{1}{2} \int_0^T \beta(t) dt} \sqrt{\mathbb{E}_{p_{\tiny \mbox{data}}(\cdot)}|X|^2/2}
+ \frac{\varepsilon}{2} \sqrt{T} + \frac{1}{2} \sqrt{\mathbb{E}_{p_{\tiny \mbox{pre}}(\cdot)} \left[e^{\frac{2 r(Y)}{\alpha}} \right]}.
\end{equation}
\item
The second term on the right side of \eqref{eq:TVd} quantifies the difference
between the distribution $p_{\tiny \mbox{pre}}(\cdot)$ of samples 
generated from the pretrained model 
and the fine-tuned distribution $p_{\tiny \mbox{ftune}}(\cdot)$.
If the reward $r(x)$ is bounded, say by $\overline{r}>0$, then the term is bounded from above by
$\frac{1}{2} \exp\left(\frac{\overline{r}}{\alpha} \right)$.
In this case, the bound \eqref{eq:VP1} for VP becomes:
\begin{equation}
D_{TV}(p_{\tiny \mbox{ftune}}(\cdot), p_{\tiny \mbox{data}}(\cdot)) \le
e^{-\frac{1}{2} \int_0^T \beta(t) dt} \sqrt{\mathbb{E}_{p_{\tiny \mbox{data}}(\cdot)}|X|^2/2}
+ \frac{\varepsilon}{2} \sqrt{T} + \frac{1}{2} e^{\frac{\overline{r}}{\alpha}}.
\end{equation}
When $\alpha \to 0$, the bound increases to $+\infty$,
which gives an explanation of the reward collapse.
\end{itemize}

\begin{proof}[Proof of Proposition \ref{prop:TVd}]
We have:
\begin{equation}
\label{eq:KLpreft}
D_{KL}(p_{\tiny \mbox{ftune}}(\cdot), p_{\tiny \mbox{pre}}(\cdot))
= - \mathbb{E}_{p_{\tiny \mbox{pre}}(\cdot)}\left[ \frac{r(Y)}{\alpha} \right]
+ \log \mathbb{E}_{p_{\tiny \mbox{pre}}(\cdot)}[e^{r(Y)/\alpha}].
\end{equation}
Note that $x \to \log x + \frac{x^2}{2}$ is convex on $[1, +\infty)$.
By Jensen's inequality,
\begin{equation}
\label{eq:Jensen}
\log \mathbb{E}_{p_{\tiny \mbox{pre}}(\cdot)}[e^{r(Y)/\alpha}] + \frac{1}{2} \left(\mathbb{E}_{{p_{\tiny \mbox{pre}}(\cdot)}} [e^{r(Y)/\alpha}]\right)^2
\le \mathbb{E}_{p_{\tiny \mbox{pre}}(\cdot)} \left( \frac{r(Y)}{\alpha} + \frac{1}{2} e^{2 r(Y)/\alpha}\right).
\end{equation}
Combining \eqref{eq:KLpreft} and \eqref{eq:Jensen} yields
\begin{equation*}
D_{KL}(p_{\tiny \mbox{ftune}}(\cdot), p_{\tiny \mbox{pre}}(\cdot))
\le \frac{1}{2} \var_{p_{\tiny \mbox{pre}}(\cdot)}(e^{r(Y)/\alpha})
\le \frac{1}{2} \mathbb{E}_{p_{\tiny \mbox{pre}}}(e^{2 r(Y)/\alpha}).
\end{equation*}
Further by Pinsker's inequality (see \cite[Theorem 7.10]{PW23}), 
we get 
$D_{TV}(p_{\tiny \mbox{ftune}}(\cdot), p_{\tiny \mbox{pre}}(\cdot)) \le \frac{1}{2} \sqrt{\mathbb{E}_{p_{\tiny \mbox{pre}}(\cdot)} \left[e^{\frac{2 r(Y)}{\alpha}} \right]}$.
The bound \eqref{eq:TVd} follows by the triangle inequality.
\end{proof}

\subsection{Stochastic control}
\label{sc32}

Now the goal is to emulate the fine-tuned distribution $p_{\tiny \mbox{ftune}}(\cdot)$ specified by \eqref{eq:cfsol}.
We follow the idea of \cite{UZ24} to lift the problem to the process level
via stochastic control:
\begin{equation}
\label{eq:diffYct}
dY_t = (\overline{b}(t, Y_t) + u_t) dt + \overline{\sigma}(t) dB_t, \quad
Y_0 \sim \nu(\cdot),
\end{equation}
where $\overline{b}(\cdot, \cdot)$ is defined in \eqref{eq:diffY},
$\overline{\sigma}(t):= \sigma(T-t)$,
and $(u_t, \, t \ge 0)$ is adapted. 
Here both $u$ and $\nu(\cdot)$ 
are decision variables,
and 
$\nu(\cdot)$ is not (necessarily) the noise distribution $p_{\tiny \mbox{noise}}(\cdot)$ as in \eqref{eq:diffY}.

\quad Denote by $Q^{u, \nu}_{[0,T]}(\cdot)$ the probability distribution of the process $(Y_t, \, 0 \le t \le T)$ defined by \eqref{eq:diffYct},
and $Q^{u, \nu}_t(\cdot)$ its marginal distribution at time $t$.
Recall that $Q_{[0,T]}(\cdot)$ is the distribution of the pretrained model \eqref{eq:diffY},
so
\begin{equation*}
Q_{[0,T]}(\cdot) = Q^{0, p_{\tiny \mbox{noise}}}_{[0,T]}(\cdot).
\end{equation*}
It is natural to consider the following objective for the stochastic control problem:
\begin{equation}
\label{eq:entproc}
(u^*, \nu^*):= \argmax_{(u, \nu) \in \mathcal{A}} \mathbb{E}_{Q^{u, \nu}_{[0,T]}(\cdot)}[r(Y_T)] - \alpha D_{KL}(Q^{u, \nu}_{[0,T]}(\cdot),Q_{[0,T]}(\cdot)),
\end{equation}
where $\mathcal{A}$ is an admissible set for the decision variables that we will specify. 
As we will see, 
$Y_T \sim p_{\tiny \mbox{ftune}}(\cdot)$ under $Q^{u^*, \nu^*}_{[0,T]}(\cdot)$
(which may not hold for the regularization by general $f$-divergence.)

\quad The following result elucidates the form of the objective \eqref{eq:entproc}.
\begin{proposition}
\label{prop:equiv}
Let 
\begin{equation}
\begin{aligned}
\mathcal{A} = \bigg\{(u, \nu):  \nu \mbox{ has a smooth density}, \,\,
\exp &\left(-\int_0^t  \frac{u_s}{\overline{\sigma}(s)} dB_s - \frac{1}{2} \int_0^t \left|\frac{u_s}{\overline{\sigma}(s)}\right|^2 ds \right) \\
& \mbox{and } \int_0^t  \frac{u_s}{\overline{\sigma}(s)} dB_s  \mbox{ are } Q^{u, \nu}_{[0,T]}\mbox{-martingales} \bigg\}.
\end{aligned}
\end{equation}
Then for each $(u, \nu) \in \mathcal{A}$, 
\begin{equation}
\label{eq:objid}
\begin{aligned}
\mathbb{E}_{Q^{u, \nu}_{[0,T]}(\cdot)}[r(Y_T)] &- \alpha D_{KL}(Q^{u, \nu}_{[0,T]}(\cdot),Q_{[0,T]}(\cdot)) \\
&= \mathbb{E}_{Q^{u, \nu}_{[0,T]}(\cdot)}\left[r(Y_T) - \frac{\alpha}{2} \int_0^T \left|\frac{u_t}{\overline{\sigma}(t)}\right|^2 dt -
\alpha \log \left( \frac{\nu(Y_0)}{p_{\tiny \mbox{noise}}(Y_0)}\right)  \right].
\end{aligned}
\end{equation}
\end{proposition}
\begin{proof}
For $(u, v) \in \mathcal{A}$, 
we have $\exp \left(-\int_0^t  \frac{u_s}{\overline{\sigma}(s)} dB_s - \frac{1}{2} \int_0^t \left|\frac{u_s}{\overline{\sigma}(s)}\right|^2 ds \right)$ is a $Q^{u, \nu}_{[0,T]}$-martingale.
So by Girsanov's theorem (\cite[Chapter 3, Theorem 5.1]{KS91}), 
\begin{equation*}
\frac{dQ^{u, \nu}_{[0,T]}}{dQ_{[0,T]}}(Y) =  \frac{\nu(Y_0)}{p_{\tiny \mbox{noise}}(Y_0)} \exp \left(\int_0^T  \frac{u_t}{\overline{\sigma}(t)} dB_t + \frac{1}{2} \int_0^T \left|\frac{u_t}{\overline{\sigma}(t)}\right|^2 dt \right).
\end{equation*}
As a result,
\begin{equation*}
\begin{aligned}
D_{KL}(Q^{u, \nu}_{[0,T]}(\cdot), Q_{[0,T]}(\cdot))
& = \mathbb{E}_{Q^{u, \nu}_{[0,T]}(\cdot)}\left[ \log \frac{dQ^{u, \nu}_{[0,T]}}{dQ_{[0,T]}}\right] \\
& =  \mathbb{E}_{Q^{u, \nu}_{[0,T]}(\cdot)}\left[ \log \left( \frac{\nu(Y_0)}{p_{\tiny \mbox{noise}}(Y_0)}\right) + \int_0^T  \frac{u_t}{\overline{\sigma}(t)} dB_t + \frac{1}{2} \int_0^T \left|\frac{u_t}{\overline{\sigma}(t)}\right|^2 dt \right] \\
& =  \mathbb{E}_{Q^{u, \nu}_{[0,T]}(\cdot)}\left[ \log \left( \frac{\nu(Y_0)}{p_{\tiny \mbox{noise}}(Y_0)}\right)
+ \frac{1}{2} \int_0^T \left|\frac{u_t}{\overline{\sigma}(t)}\right|^2 dt \right],
\end{aligned}
\end{equation*}
where the last equation is due to the fact that $\int_0^t  \frac{u_s}{\overline{\sigma}(s)} dB_s$ is a $Q^{u, \nu}_{[0,T]}$-martingale.
This yields the identity \eqref{eq:objid}. 
\end{proof}

\quad The condition that $\exp \left(-\int_0^t  \frac{u_s}{\overline{\sigma}(s)} dB_s - \frac{1}{2} \int_0^t \left|\frac{u_s}{\overline{\sigma}(s)}\right|^2 ds \right)$ and 
$\int_0^t  \frac{u_s}{\overline{\sigma}(s)} dB_s$ are $Q^{u, \nu}_{[0,T]}$-martingales
allows to simplify the expression of $D_{KL}(Q^{u, \nu}_{[0,T]}(\cdot),Q_{[0,T]}(\cdot))$,
bypassing the local martingale trap.
Sufficient conditions are known for these local martingales to be (true) martingales.
For instance, 
Novikov's condition $\mathbb{E}_{Q^{u, \nu}_{[0,T]}(\cdot)}\left[\exp \left(\frac{1}{2} \int_0^T \left|\frac{u_t}{\overline{\sigma}(t)}\right|^2 dt  \right)\right] < \infty$
ensures the use of Girsanov's theorem,
and also the $Q^{u, \nu}_{[0,T]}$-martingale property of $\int_0^t  \frac{u_s}{\overline{\sigma}(s)} dB_s$.

\quad By Proposition \ref{prop:equiv}, 
the stochastic control problem \eqref{eq:entproc} is rewritten as:
\begin{equation}
\label{eq:entproc2}
\begin{aligned}
(u^*, \nu^*):&= \argmax_{(u, \nu) \in \mathcal{A}} \mathbb{E}_{Q^{u, \nu}_{[0,T]}(\cdot)}\left[r(Y_T) - \frac{\alpha}{2} \int_0^T \left|\frac{u_t}{\overline{\sigma}(t)}\right|^2 dt -
\alpha \log \left( \frac{\nu(Y_0)}{p_{\tiny \mbox{noise}}(Y_0)}\right)  \right] \\
& = \argmax_{(u, \nu) \in \mathcal{A}}
\int \mathbb{E}_{Q^{u, y}_{[0,T]}(\cdot)}\left[r(Y_T) - \frac{\alpha}{2} \int_0^T \left|\frac{u_t}{\overline{\sigma}(t)}\right|^2 dt \right] \nu(y) dy - \alpha D_{KL}(\nu(\cdot), p_{\tiny \mbox{noise}}(\cdot)),
\end{aligned}
\end{equation}
where $Q^{u, y}_{[0,T]}(\cdot)$ denotes the probability distribution of the process $(Y_t, \, 0 \le t \le T)$ in \eqref{eq:diffYct} conditioned on $Y_0 = y$.
It is easy to see that the two decision variables
 $(u, \nu)$ are separable in the problem \eqref{eq:entproc2}.
It boils down to the following two steps:
\begin{enumerate}[itemsep = 3 pt]
\item
Solve the (standard) stochastic control problem:
\begin{equation}
\label{eq:optu}
v^*(0,y):= \max_{u \in \mathcal{A}} \mathbb{E}_{Q^{u, y}_{[0,T]}(\cdot)}\left[r(Y_T) - \frac{\alpha}{2} \int_0^T \left|\frac{u_t}{\overline{\sigma}(t)}\right|^2 dt \right],
\end{equation}
to get the optimal control $u^*(\cdot, \cdot)$.
\item
Given $v^*(0,y)$, solve for the optimal initial distribution:
\begin{equation}
\label{eq:optnu}
\nu^*(\cdot):= \argmax_\nu \mathbb{E}_{\nu(\cdot)}[v^*(0, Y)] - \alpha D_{KL}(\nu(\cdot), p_{\tiny \mbox{noise}}(\cdot)).
\end{equation}
\end{enumerate}

\subsection{Solve the stochastic control problem}
\label{sc33}

Here we study the stochastic control problem \eqref{eq:entproc} (or equivalently \eqref{eq:entproc2}),
following \eqref{eq:optu}--\eqref{eq:optnu}.
We start with the problem \eqref{eq:optnu} to find the optimal initial distribution $\nu^*(\cdot)$,
the easier one.
\begin{proposition}
\label{prop:optnu}
Assume that $\int \exp \left( \frac{v^*(0,y)}{\alpha}\right) p_{\tiny \mbox{noise}}(y) dy < \infty$.
We have:
\begin{equation}
\nu^*(x) = \frac{1}{C'} \exp \left( \frac{v^*(0,y)}{\alpha}\right) p_{\tiny \mbox{noise}}(y),
\end{equation}
where $C' = \int \exp \left( \frac{v^*(0,y)}{\alpha}\right) p_{\tiny \mbox{noise}}(y) dy$ is the normalizing constant.
\end{proposition}
\begin{proof}
The proof is the same as  Lemma \ref{lem:DVvar}, 
by replacing $r(y)$ with $v^*(0, y)$ and $p_{\tiny \mbox{pre}}(y)$ with $p_{\tiny \mbox{noise}}(y)$.
\end{proof}

\quad As we will see, the assumption $\int \exp \left( \frac{v^*(0,y)}{\alpha}\right) p_{\tiny \mbox{noise}}(y) dy < \infty$
is equivalent to 
$\int \exp \left( \frac{r(y)}{\alpha}\right) p_{\tiny \mbox{pre}}(y) dy < \infty$ in Lemma \ref{lem:DVvar}.

\quad Now we proceed to solving the problem \eqref{eq:optu}.
To this end, define the value-to-go:
\begin{equation}
\label{eq:vt}
v^*(t,y):= \max_{u \in \mathcal{A}_t} \mathbb{E}_{Q^{u, y}_{[t,T]}(\cdot)}\left[r(Y_T) - \frac{\alpha}{2} \int_t^T \left|\frac{u_s}{\overline{\sigma}(s)}\right|^2 ds \right],
\end{equation}
where $Q^{u, y}_{[t,T]}(\cdot)$ denotes the probability distribution of $(Y_s, \, t \le s \le T)$ in \eqref{eq:diffYct}  conditioned on
$Y_t = y$,
and $\mathcal{A}_t$ is the admissible set for $u$ on $[t, T]$.
The following result provides a verification theorem for the problem \eqref{eq:vt}.
\begin{proposition}
\label{prop:HJ}
Assume that there exists $V(t,y) \in C^{1,2}([0,T] \times \mathbb{R}^d)$ 
of at most polynomial growth in $y$,
which solves the Hamilton-Jacobi equation:
\begin{equation}
\label{eq:HJ}
\frac{\partial v}{\partial t} + \frac{\overline{\sigma}^2(t)}{2} \Delta v + \overline{b}(t,y) \cdot \nabla v + \frac{\overline{\sigma}^2(t)}{2 \alpha} |\nabla v|^2 = 0, \quad v(T,y) = r(y),
\end{equation}
and that $\frac{\overline{\sigma}^2(t)}{\alpha} \nabla V(\cdot,\cdot) \in \mathcal{A}$.
Then $v^*(t,y) = V(t,y)$, 
and the feedback control $u^*(t,y) = \frac{\overline{\sigma}^2(t)}{\alpha} \nabla V(t,y)$.
\end{proposition}
\begin{proof}
Let's first derive the Hamilton-Jacobi equation \eqref{eq:HJ}.
Dynamic program shows that $v^*(t,y)$ solves (in some sense) that Hamilton-Jacobi-Bellman equation:
\begin{equation}
\label{eq:HJB}
\max_u \left\{ \frac{\partial v}{\partial t} + \frac{\overline{\sigma}^2(t)}{2} \Delta v + (\overline{b}(t,y) + u) \cdot \nabla v - \frac{\alpha |u|^2}{2 \overline{\sigma}^2(t)}\right\} = 0.
\end{equation}
The maximum is attained at $u^*(t,y) = \frac{\overline{\sigma}^2(t)}{\alpha} \nabla v(t,y)$ (if admissible).
Plugging it back into \eqref{eq:HJB} yields the equation \eqref{eq:HJ}.
By standard verification argument (see \cite{YZ99}),
if a smooth function $V(t,y)$ solves the Hamilton-Jacobi equation \eqref{eq:HJ}
with $U(t, y):=\frac{\overline{\sigma}^2(t)}{\alpha} \nabla V(t,y)$ admissible,
then $v^* = V$ and $u^* = U$.
\end{proof}

\quad It is preferable to characterize the value function $v^*(t,y)$ as
the viscosity solution to the Hamilton-Jacobi equation \eqref{eq:HJ}.
Assuming that $\overline{b}(\cdot, \cdot)$, $\sigma(\cdot)$ and $r(\cdot)$ are bounded and Lipschitz,
the equation \eqref{eq:HJB} has a unique continuous viscosity solution by adapting the arguments in \cite{AT15}.
By \cite{PP13}, this solution is Lipschitz continuous, which can further be shown to be smooth \cite{Kry87}.
So Proposition \ref{prop:HJ} can be applied.

\quad The purpose of the stochastic control \eqref{eq:entproc} (or \eqref{eq:entproc2})
is to fine-tune the pretrained model for data generation.
We consider the distribution of $(Y_t, \, 0 \le t \le T)$ under the optimal control $Q^{u^*, \nu^*}_{[0,T]}(\cdot)$.
\begin{proposition}
\label{prop:keyd}
Let $(u^*, \nu^*)$ be the optimal solution to the problem \eqref{eq:entproc} (or \eqref{eq:entproc2}),
and let the assumptions of Proposition \ref{prop:optnu} and \ref{prop:HJ} hold.
We have:
\begin{equation}
\label{eq:Quvstar}
Q^{u^*, \nu^*}_{[0,T]}(dY_{\bullet})
= \frac{1}{C} \exp\left( \frac{r(Y_T)}{\alpha}\right) Q_{[0,T]}(dY_\bullet),
\end{equation}
where $Y_\bullet = (Y_t, \, 0 \le t \le T)$.
Consequently, the marginal distribution is 
\begin{equation}
\label{eq:Quvt}
Q^{u^*, \nu^*}_{t}(y) = \frac{1}{C} \mathbb{E}_{Q^y_{[t,T]}(\cdot)}\left[ \exp\left(\frac{r(Y_T)}{\alpha} \right)\right] Q_t(y), 
\quad 0 \le t \le T,
\end{equation}
where $Q^y_{[t,T]}(\cdot)$ denotes the probability distribution of $(Y_s, \, t \le s \le T)$ under $Q_{[t,T]}(\cdot)$
conditioned on $Y_t = y$.
\end{proposition}
\begin{proof}
Recall that $Q^{u^*, y}_{[0,T]}(\cdot)$ denotes the probability distribution of \eqref{eq:diffYct} conditioned on $Y_0 = y$.
We first show that
\begin{equation}
\label{eq:Qcond}
Q^{u^*, y}_{[0,T]}(y_\bullet) = \frac{\exp\left(\frac{r(Y_T)}{\alpha}\right) Q^y_{[0,T]}(Y_\bullet)}{\mathbb{E}_{Q^y_{[0,T]}(\cdot)}\left[ \exp\left(\frac{r(Y_T)}{\alpha}\right)\right]}.
\end{equation} 
Note that 
\begin{equation}
\label{eq:KLcond12}
\begin{aligned}
&\quad D_{KL}\left(Q^{u^*, y}_{[0,T]}(Y_\bullet),  \frac{\exp\left(\frac{r(Y_T)}{\alpha}\right) Q^y_{[0,T]}(Y_\bullet)}{\mathbb{E}_{Q^y_{[0,T]}(\cdot)}\left[ \exp\left(\frac{r(Y_T)}{\alpha}\right)\right]}\right) \\
&  = D_{KL}(Q^{u^*, y}_{[0,T]}(Y_\bullet), Q^y_{[0,T]}(Y_\bullet)) 
- \mathbb{E}_{Q^y_{[0,T]}(\cdot)}\left\{ \frac{r(Y_T)}{\alpha} - \log \mathbb{E}_{Q^y_{[0,T]}(\cdot)}\left[ \exp\left(\frac{r(Y_T)}{\alpha}\right)\right] \right\} \\
& =  \mathbb{E}_{Q^y_{[0,T]}(\cdot)}\left[ \frac{1}{2} \int_0^T \left| \frac{u^*(t, Y_t)}{\overline{\sigma}(t)}\right|^2dt -\frac{r(Y_T)}{\alpha} \right] +\log \mathbb{E}_{Q^y_{[0,T]}(\cdot)}\left[ \exp\left(\frac{r(Y_T)}{\alpha}\right)\right] \\
& = - v^*(0, y) + \log \mathbb{E}_{Q^y_{[0,T]}(\cdot)}\left[ \exp\left(\frac{r(Y_T)}{\alpha}\right)\right].
\end{aligned}
\end{equation}
where the first equation follows from the chain rule for KL divergence,
and the second equation is from Girsanov's theorem.
Recall that $v^*(t,y)$ solves the Hamilton-Jacobi equation \eqref{eq:HJ}.
Let $\mathcal{V}(t,y): = \exp(v^*(t,y)/\alpha)$. 
By It\^o's formula, $\mathcal{V}(t,y)$ solves the linear equation:
\begin{equation*}
\frac{\partial \mathcal{V}}{\partial t} + \frac{\overline{\sigma}^2(t)}{2} \Delta \mathcal{V} + \overline{b}(t, y) \cdot \mathcal{V} = 0,
\quad \mathcal{V}(T,y) = \exp \left( \frac{r(y)}{\alpha} \right).
\end{equation*}
The Feynman-Kac formula (\cite[Chapter 4, Theorem 4.2]{KS91}) yields:
\begin{equation}
\label{eq:FK}
\exp \left( \frac{v^*(t,y)}{\alpha}\right) = \mathbb{E}_{Q^y_{[t,T]}(\cdot)}\left[ \exp\left(\frac{r(Y_T)}{\alpha}\right)\right].
\end{equation}
By setting $t = 0$ in \eqref{eq:FK} and injecting it into \eqref{eq:KLcond12},
we get
$D_{KL}\left(Q^{u^*, y}_{[0,T]}(y_\bullet),  \frac{\exp\left(\frac{r(Y_T)}{\alpha}\right) Q^y_{[0,T]}(Y_\bullet)}{\mathbb{E}_{Q^y_{[0,T]}(\cdot)}\left[ \exp\left(\frac{r(Y_T)}{\alpha}\right)\right]}\right) =0$
so \eqref{eq:Qcond} is proved. 

\quad Next by Proposition \ref{prop:optnu},
we have:
\begin{equation}
\label{eq:nustar2}
\nu^*(y) = \frac{1}{C'} \exp\left( \frac{v^*(0,y)}{\alpha}\right) p_{\tiny \mbox{noise}}(y)
= \frac{1}{C'} \mathbb{E}_{Q^y_{[0,T]}(\cdot)}\left[ \exp\left(\frac{r(Y_T)}{\alpha}\right)\right] p_{\tiny \mbox{noise}}(y).
\end{equation}
where the second equation, again, follows \eqref{eq:FK}.
Furthermore, 
\begin{equation}
\label{eq:CC}
C' := \int \exp \left( \frac{v^*(0,y)}{\alpha}\right) p_{\tiny \mbox{noise}}(y) dy
= \mathbb{E}_{Q_{[t,T]}(\cdot)}\left[ \exp\left(\frac{r(Y_T)}{\alpha}\right)\right]
= \mathbb{E}_{p_{\tiny \mbox{pre}}(\cdot)}\left[ \exp\left(\frac{r(Y)}{\alpha}\right)\right] =: C.
\end{equation}
Combining \eqref{eq:Qcond}, \eqref{eq:nustar2} and \eqref{eq:CC} yields \eqref{eq:Quvstar}.
The identity \eqref{eq:Quvt} is obtained by marginalizing over $t$.
\end{proof}

\quad By specializing \eqref{eq:Quvt} at $t = T$, 
we get $Q^{u^*, \nu^*}_{T}(\cdot) = p_{\tiny \mbox{ftune}}(\cdot)$.
Indeed, the stochastic control problem \eqref{eq:entproc} 
yields the fine-tuned distribution \eqref{eq:cfsol}.
The equality \eqref{eq:CC} shows that 
the assumptions in Lemma \ref{lem:DVvar} and Proposition \ref{prop:optnu} are equivalent.

\quad In \cite{UZ24}, the optimal control problem \eqref{eq:optu} or \eqref{eq:vt}
is solved by neural ODEs/SDEs \cite{Chen18, Kid21}.
Here we offer some other thoughts.
By \eqref{eq:FK} and Proposition \ref{prop:HJ}, 
the optimal feedback control is given by
\begin{equation}
\label{eq:utu}
u^*(t,y) = \frac{\overline{\sigma}^2(t)}{\alpha} \nabla v^*(t,y)
= \overline{\sigma}^2(t) \nabla \left(\ln \mathbb{E}_{Q^y_{[t,T]}(\cdot)}\left[ \exp\left(\frac{r(Y_T)}{\alpha}\right)\right]\right).
\end{equation}
By ``simply" exchanging $\ln$ and $\mathbb{E}_{Q^y_{[t,T]}(\cdot)}$,
we get
\begin{equation}
\widetilde{u}^*(t, y):= \frac{\overline{\sigma}^2(t)}{\alpha} \nabla \, \mathbb{E}_{Q^y_{[t,T]}(\cdot)}[r(Y_T)] \stackrel{?}{\approx} u^*(t,y),
\end{equation}
where the term $\nabla \, \mathbb{E}_{Q^y_{[t,T]}(\cdot)}[r(Y_T)]$ can further be computed 
by classifier guidance \cite{Dh21}.
An alternative approach to compute $\nabla \, \mathbb{E}_{Q^y_{[t,T]}(\cdot)}[r(Y_T)]$ is by 
Malliavin calculus \cite{FL99} (that is well-studied in the mathematical finance literature.)
To be more precise,
\begin{equation}
\nabla \, \mathbb{E}_{Q^y_{[t,T]}(\cdot)}[r(Y_T)]
= \mathbb{E}_{Q^y_{[t,T]}(\cdot)}[\nabla \,r(Y_T) \cdot Z_T],
\end{equation}
where $(Z_s, \, t \le s \le T)$ (formally) solves:
\begin{equation}
dZ_s = \nabla \overline{b}(s, Y_s) Z_s ds , \quad Z_t = I.
\end{equation}
See \cite{CG07, GM05} for a rigorous treatment.
But it is not clear how good/bad the approximation \eqref{eq:utu} is in general.

\quad To sample the optimal initial distribution $\nu^*(\cdot)$, 
we consider another stochastic control problem:
\begin{equation}
\label{eq:gennu}
dY_t = q_t dt + \sigma'(t) dB'_t, \quad Y_0 \sim p_{\tiny \mbox{fix}}(\cdot),
\end{equation}
where $q_t$ is the control variable,
$\sigma': \mathbb{R}_+ \to \mathbb{R}_+$,
and
the distribution $p_{\tiny \mbox{fix}}(\cdot)$ is set so that 
the process
$dY'_t = \sigma'(t) dB'_t$, $Y'_0 \sim p_{\tiny \mbox{fix}}(\cdot)$
will generate $Y'_T \sim p_{\tiny \mbox{noise}}(\cdot)$ at time $T$.
Denote by $P^{q}_{[0,T]}(\cdot)$ be the probability distribution of $(Y_t, \, 0 \le t \le T)$,
defined by \eqref{eq:gennu}.,
and $P^q_t$ its marginal distribution at time $t$.
Solve the optimization problem:
\begin{equation}
q^*: = \argmax_q \mathbb{E}_{P^q_{[0,T]}(\cdot)}\left[v^*(0, Y_0) - \frac{\alpha}{2} \int_0^T  \left|\frac{q_t}{\sigma'(t)} \right|^2 dt\right],
\end{equation}
and it is easy to see that $P^{q^*}_T(\cdot) = \nu^*(\cdot)$.

\quad The following result quantifies the performance of generating $p_{\tiny \mbox{ftune}}(\cdot)$
using any approximation $(\widetilde{u}^*, \widetilde{\nu}^*)$.
\begin{proposition}
Define $\eta^2:= \max_{0 \le t \le T} \mathbb{E}_{Q^{u^*, \nu^*}_t(\cdot)} |\widetilde{u}^*(t, Y) - u^*(t,Y)|^2$.
We have:
\begin{equation}
\label{eq:TVbad}
D_{TV}(Q^{\widetilde{u}^*, \nu^*}_{T}(\cdot), p_{\tiny \mbox{ftune}}(\cdot)) \le \frac{\eta}{2} \sqrt{\int_0^T \frac{1}{\overline{\sigma}^2(t)} dt} + \sqrt{\frac{1}{2} D_{KL}(\widetilde{\nu}^*(\cdot), \nu^*(\cdot))}.
\end{equation}
\end{proposition}
\begin{proof}
By the data processing inequality,
we get
\begin{equation*}
\begin{aligned}
D_{KL}(p_{\tiny \mbox{ftune}}(\cdot), Q^{\widetilde{u}^*, \widetilde{\nu}^*}_{T}(\cdot)) 
& = D_{KL}(Q^{u^*, \nu^*}_{T}(\cdot), Q^{\widetilde{u}^*, \widetilde{\nu}^*}_{T}(\cdot)) \\
& \le  D_{KL}(Q^{u^*, \nu^*}_{[0,T]}(\cdot), Q^{\widetilde{u}^*, \widetilde{\nu}^*}_{[0,T]}(\cdot)) \\
& = \frac{1}{2} \mathbb{E}_{Q^{u^*, \nu^*}_{[0,T]}(\cdot)}\left(\int_0^T \left| \frac{u^*(t, Y_t) - \widetilde{u}^*(t,Y_t)}{\overline{\sigma}(t)} \right|^2  dt\right) + D_{KL}(\widetilde{\nu}^*(\cdot), \nu^*(\cdot)) \\
& \le \frac{\eta^2}{2} \int_0^T \frac{1}{\overline{\sigma}^2(t)} dt +  D_{KL}(\widetilde{\nu}^*(\cdot), \nu^*(\cdot)).
\end{aligned}
\end{equation*}
Applying Pinsker's inequality yields \eqref{eq:TVbad}.
\end{proof}

\section{Extension to regularization by $f$-divergence}
\label{sc4}

\quad In this section, we go beyond entropy regularization,
and consider the problem of fine-tuning regularized by $f$-divergence.

\quad Let $f: \mathbb{R}_+ \to (-\infty, +\infty]$
be a convex function such that 
$f(x) < \infty$ for $x > 0$,
$f(1) = 0$ and $f(0) = \lim_{t \to 0^+} f(t)$.
Given two probability distribution $p(\cdot)$, $q(\cdot)$ on suitably nice space, 
the $f$-divergence between $p(\cdot)$ and $q(\cdot)$ is defined by:
\begin{equation}
\label{eq:Df}
D_f(P, Q):= \int f\left(\frac{dp}{dq}\right) dq = \mathbb{E}_{q(\cdot)}\left[ f\left(\frac{dp}{dq}\right)\right].
\end{equation}
The KL divergence corresponds to the choice $f(t) = t \ln t$.
Other popular examples include:
\begin{itemize}[itemsep = 3 pt]
\item
Forward KL: $f(t) = -\ln t$.
\item
$\gamma$-divergence: $f(t) = \frac{t^{1 - \gamma} - (1 - \gamma)t -\gamma}{\gamma(\gamma - 1)}$.
\item
Total variation distance: $f(t) = \frac{1}{2}|t-1|$.
\end{itemize}
The $\gamma$-divergence becomes forward KL when $\gamma \to 1$,
and KL when $\gamma \to 0$.
So the KL divergence is also referred to as the reverse KL.
See \cite[Chapter 7]{PW23} 
for an introduction to the $f$-divergence.

\quad Inspired by entropy-regularized fine-tuning \eqref{eq:entFT},
an obvious candidate is:
\begin{equation}
\label{eq:entfFT}
p^f_{\star}(\cdot):=
\argmax_p \mathbb{E}_{p(\cdot)}[r(Y)] - \alpha D_{f}(p(\cdot), p_{\tiny \mbox{pre}}(\cdot)),
\end{equation}
where the maximization is over all probability distribution on $\mathbb{R}^d$.
Similar to entropy-regularized fine-tuning \eqref{eq:entFT},
the problem \eqref{eq:entfFT} is solvable under some conditions on $f$.
The following result was proved by \cite{WJ24} 
in a different but closely related context 
of direct preference optimization \cite{RS23}.
For ease of reference, we record the proof.
\begin{proposition}
\label{prop:pfftune}
Assume that $p_{\tiny \mbox{pre}}(\cdot)$ has full support on $\mathbb{R}^d$,
and $f'$ is strictly increasing, and $0$ is not in the domain of $f'$.
Also assume that 
there exists $\lambda$ such that
$\frac{r(\cdot) - \lambda}{\alpha}$ is in the range of $f'$,
and 
$\int (f')^{-1}\left(\frac{r(y) - \lambda}{\alpha} \right) p_{\tiny \mbox{pre}}(y) dy =1$.
We have:
\begin{equation}
\label{eq:pfftune}
p^f_{\star}(y) = (f')^{-1}\left(\frac{r(y) - \lambda}{\alpha} \right) p_{\tiny \mbox{pre}}(y).
\end{equation}
\end{proposition}
\begin{proof}
Define the Lagrangian 
\begin{equation*}
\mathcal{L}(p, \lambda, \alpha):=
\mathbb{E}_{p(\cdot)}[r(Y)] - \alpha \mathbb{E}_{p_{\tiny \mbox{pre}}(\cdot)}\left[ f\left( \frac{p(Y)}{p_{\tiny \mbox{pre}}(Y)}\right)\right] - \lambda \left( \int \pi(y) dy - 1\right) + \mathbb{E}_{p(\cdot)}[\kappa(Y)],
\end{equation*}
where $p$ is the primal variable, and $\lambda$, $\kappa(y)$ are dual variables. 
By the KKT condition (see \cite[Section 5.5]{BV04}), we get:
\begin{itemize}[itemsep = 3 pt]
\item
First order condition: $\nabla_p \mathcal{L}(p, \lambda, \alpha) = 0$, which yields
\begin{equation}
\label{eq:KKTstat}
r(y) - \alpha f'\left( \frac{p(y)}{p_{\tiny \mbox{pre}}(y)}\right) - \lambda + \kappa(y) = 0, \quad \mbox{for all } y.
\end{equation}
\item
Primal feasibility: $\int p(y) dy = 1$ and $p(y) \ge 0$.
\item
Dual feasibility: $\kappa(y) \ge 0$ for all $y$.
\item
Complementary slackness: $\kappa(y) p(y) = 0$ for all $y$.
\end{itemize}
Since $f'$ is strictly increasing (invertible), we obtain by \eqref{eq:KKTstat}:
\begin{equation*}
p(y) = (f')^{-1}\left(\frac{r(y) - \lambda + \kappa(y)}{\alpha} \right) p_{\tiny \mbox{pre}}(y).
\end{equation*}
Because $p_{\tiny \mbox{pre}}(y) > 0$ and $0$ is not in the domain of $f'$,
we have $p(y) > 0$ for all $y$.
Thus, the complementary slackness implies $\kappa(y) = 0$.
The constant $\lambda$ is determined by the primal condition.
This yield the formula \ref{eq:pfftune}.
\end{proof}

\quad It was claimed in \cite[p.5]{WJ24} that the solution \eqref{eq:pfftune}
to the (constrained) problem \eqref{eq:entfFT}
can be written in the form 
$\frac{1}{C_f} (f')^{-1}\left(\frac{r(\cdot)}{\alpha}\right) p_{\tiny \mbox{pre}}(\cdot)$.
This seems to be wrong because
$\frac{r(\cdot)}{\alpha}$ may even not lie in the range of $f'$.
In fact, $\lambda$ in \eqref{eq:pfftune} plays the role of the normalizing constant, 
and can be factorized out only for entropy regularization 
where $(f')^{-1}(t) = e^{t - 1}$.
Sampling the distribution \eqref{eq:pfftune}
is generally difficult due to the complex dependence in the normalizing constant $\lambda$.

\quad As an alternative, we can first solve the problem \eqref{eq:entfFT} with no constraint,
and then normalize it to a probability distribution
(which was claimed in \cite{WJ24} to be the ``solution".)
This leads to the following fine-tuning proposal.
Recall that $r: \mathbb{R}^d \to \mathbb{R}_+$ is assumed to be nonnegative. 
\begin{definition}
\label{def:fdiv}
Assume that $(f')^{-1}(\cdot)$ is well-defined on the range of $\frac{r(\cdot)}{\alpha}$,
or $(f')^{-1}(\cdot)$ is not but $(-f')^{-1}(\cdot)$ is well-defined on the range of $\frac{r(\cdot)}{\alpha}$,
which we denote by $(f_{\pm}')^{-1}(\cdot)$.
Also assume that 
$\int (f'_{\pm})^{-1}\left(\frac{r(y)}{\alpha}\right) p_{\tiny \mbox{pre}}(y)dy < \infty$.
Define
\begin{equation}
\label{eq:fdivft}
p^f_{\tiny \mbox{ftune}}(y) = \frac{1}{C_f} (f'_\pm)^{-1}\left( \frac{r(y)}{\alpha}\right) p_{\tiny \mbox{pre}}(y),
\end{equation}
where $C_f: = \int (f'_{\pm})^{-1}\left(\frac{r(y)}{\alpha}\right) p_{\tiny \mbox{pre}}(y)dy$
is the normalizing constant.
We abuse by calling \eqref{eq:fdivft} (instead of \eqref{eq:entfFT})
fine-tuning regularized by $f$-divergence.
\end{definition}

\quad Let's illustrate Definition \ref{def:fdiv} with examples.
\begin{itemize}[itemsep = 3 pt]
\item
KL divergence: $f'(t) = \ln t + 1$.
So $p^f_{\tiny \mbox{ftune}}(y) \propto \exp\left(\frac{r(y)}{\alpha} \right) p_{\tiny \mbox{pre}}(y)$,
and the definitions \eqref{eq:entfFT} and \eqref{eq:fdivft} agree.
\item
Forward KL: $f'(t) = -\frac{1}{t}$ so $(f'_{\pm})^{-1}(t) = (- f')^{-1}(t) = \frac{1}{t}$.
We have $p^f_{\tiny \mbox{ftune}}(y) \propto \alpha p_{\tiny \mbox{pre}}(y)/r(y)$.
\item
$\gamma$-divergence: $f'(t) = (1 - t^{-\gamma})/\gamma$. 
So if $r(\cdot)$ is bounded from below by $\alpha /\gamma$, we get
$(f'_\pm)^{-1}(t) = (-f')^{-1}(t) = (\gamma t - 1)^{-\frac{1}{\gamma}}$
and $p^f_{\tiny \mbox{ftune}}(y) \propto \left(\frac{\gamma\, r(y)}{\alpha} - 1\right)^{-\frac{1}{\gamma}} p_{\tiny \mbox{pre}}(y)$;
if $r(\cdot)$ is bounded from above by $\alpha /\gamma$,
we get 
$(f'_\pm)^{-1}(t) = (f')^{-1}(t) = (1 - \gamma t)^{-\frac{1}{\gamma}}$
and 
$p^f_{\tiny \mbox{ftune}}(y) \propto \left(1 - \frac{\gamma\, r(y)}{\alpha} \right)^{-\frac{1}{\gamma}} p_{\tiny \mbox{pre}}(y)$.
\item
Total variation distance: $(f'_\pm)^{-1}$ is not well-defined.
\end{itemize}
We point out that under Definition \ref{def:fdiv},
the fine-tuned distribution regularized by $\gamma$-divergence with $\gamma = 1$
differs from that regularized by forward KL, due to a shift in the $f$-derivatives.

\quad The advantage of the definition \eqref{eq:fdivft} is that the normalizing constant is factored out in $(f'_{\pm})^{-1}$,
so the machinery in Section \ref{sc3} can be applied 
to generate $p^f_{\tiny \mbox{ftune}}(\cdot)$
by only modifying the reward function.
Let 
\begin{equation}
\label{eq:rf}
r_f(y):= \ln \left( (f'_\pm)^{-1}\left( \frac{r(y)}{\alpha}\right) \right),
\end{equation}
so $p^f_{\tiny \mbox{ftune}}(y) = \frac{1}{C_f} \exp(r_f(y)) p_{\tiny \mbox{pre}}(y)$.
The following result is a corollary of Proposition \ref{prop:TVd}
to quantify the difference between the fine-tuned distribution \eqref{eq:fdivft}
and the original data distribution $p_{\tiny \mbox{data}}(\cdot)$.
\begin{corollary}
Assume that $\kappa: = \inf_y r_f(y) > -\infty$.
We have:
\begin{equation}
\label{eq:TVd2}
D_{TV}(p^f_{\tiny \mbox{ftune}}(\cdot), p_{\tiny \mbox{data}}(\cdot))
\le D_{TV}(p_{\tiny \mbox{pre}}(\cdot), p_{\tiny \mbox{data}}(\cdot)) + \frac{e^{-\kappa}}{2} \sqrt{\mathbb{E}_{p_{\tiny \mbox{pre}}(\cdot)}\left[ \left((f'_\pm)^{-1}\left( \frac{r(Y)}{\alpha}\right) \right)^2\right]}.
\end{equation}
\end{corollary}
\begin{proof}
It suffices to note that $x \to \log x + \frac{e^{-2 \kappa} x^2}{2}$
is convex on $[e^{\kappa}, \infty)$.
The same argument in Proposition \ref{prop:TVd} shows:
\begin{equation*}
D_{KL}(p^f_{\tiny \mbox{ftune}}(\cdot), p_{\tiny \mbox{pre}}(\cdot)) \le \frac{e^{-2 \kappa}}{2} \mathbb{E}_{p_{\tiny \mbox{pre}}(\cdot)}[e^{2 r_f(Y)}]
= \frac{e^{-2 \kappa}}{2} \mathbb{E}_{p_{\tiny \mbox{pre}}(\cdot)}\left[ \left((f'_\pm)^{-1}\left( \frac{r(Y)}{\alpha}\right) \right)^2\right].
\end{equation*}
This yields the bound \eqref{eq:TVd2}.
\end{proof}

\quad Following the stochastic control setting in Section \ref{sc3}, 
we consider the problem:
\begin{equation}
\label{eq:scf}
\begin{aligned}
(u^*, \nu^*):&= \argmax_{u, \nu) \in \mathcal{A}} \mathbb{E}_{Q^{u, \nu}_{[0,T]}(\cdot)}[r_f(Y_t)] - D_{KL}(Q^{u, \nu}_{[0,T]}(\cdot), Q_{[0,T]}(\cdot)) \\
& = \argmax_{(u, \nu) \in \mathcal{A}}
\int \mathbb{E}_{Q^{u, y}_{[0,T]}(\cdot)}\left[r_f(Y_T) - \frac{1}{2} \int_0^T \left|\frac{u_t}{\overline{\sigma}(t)}\right|^2 dt \right] \nu(y) dy -  D_{KL}(\nu(\cdot), p_{\tiny \mbox{noise}}(\cdot)),
\end{aligned}
\end{equation}
where the parameter $\alpha$ is included in $r_f(\cdot)$,
so it does not appear in the regularizer $D_{KL}(\nu(\cdot), p_{\tiny \mbox{noise}}(\cdot))$.

\quad Here we use the entropy regularizer instead of $f$-divergence,
i.e., $D_{f}(Q^{u, \nu}_{[0,T]}(\cdot), Q_{[0,T]}(\cdot))$
in the stochastic control problem.
This is because 
it is difficult to simplify $D_{f}(Q^{u, \nu}_{[0,T]}(\cdot), Q_{[0,T]}(\cdot))$ 
due to lack of the chain rule (for general $f$-divergence),
which will incur nonlocal terms.
So in the problem \eqref{eq:scf},
the ``regularization by $f$-divergence" is reflected
in the reward function $r_f(\cdot)$.
As a consequence,
Propositions \ref{prop:optnu}--\ref{prop:keyd}
are easily adapted to solve the control problem \eqref{eq:scf}.
We summarize the results in the following corollary. 
\begin{corollary}
Define the value-to-go:
\begin{equation}
v_f^*(t,y):= \max_{u \in \mathcal{A}_t} \mathbb{E}_{Q^{u, y}_{[t,T]}(\cdot)}\left[r_f(Y_T) - \frac{1}{2} \int_t^T \left|\frac{u_s}{\overline{\sigma}(s)}\right|^2 ds \right].
\end{equation}
\begin{enumerate}[itemsep = 3 pt]
\item
Assume that there exists $V(t,y) \in C^{1,2}([0,T] \times \mathbb{R}^d)$ 
of at most polynomial growth in $y$,
which solves the Hamilton-Jacobi equation:
\begin{equation}
\frac{\partial v}{\partial t} + \frac{\overline{\sigma}^2(t)}{2} \Delta v + \overline{b}(t,y) \cdot \nabla v + \overline{\sigma}^2(t) |\nabla v|^2 = 0, \quad v(T,y) = r_f(y),
\end{equation}
and that $\overline{\sigma}^2(t) \nabla V(\cdot,\cdot) \in \mathcal{A}$.
Then $v_f^*(t,y) = V(t,y)$, and the feedback control $u_f^*(t,y) = \overline{\sigma}^2(t)\nabla V(t,y)$.
\item
The optimal initial distribution is given by
\begin{equation}
\nu_f^*(x) = \frac{1}{C_f} \exp \left(v_f^*(0,y)\right) p_{\tiny \mbox{noise}}(y).
\end{equation}
\item
Let $(u^*_f, \nu^*_f)$ be the optimal solution to the problem \eqref{eq:scf}, specified by (1) and (2).
The distribution of the process under the optimal control is:
\begin{equation}
Q^{u_f^*, \nu_f^*}_{[0,T]}(dY_{\bullet})
= \frac{1}{C_f} \exp\left( r_f(Y_T)\right) Q_{[0,T]}(dY_\bullet),
\end{equation}
and in particular, $Q^{u_f^*, \nu_f^*}_{T}(\cdot) = p^f_{\tiny \mbox{ftune}}(\cdot)$.
\end{enumerate}
\end{corollary}

\quad Similar to the discussions in Section \ref{sc3},
the optimal control problem can be solved by neural ODEs/SDEs, or
be approximated by
\begin{equation}
\begin{aligned}
u_f^*(t,y) & = \overline{\sigma}^2(t) \nabla \left(\ln \mathbb{E}_{Q^y_{[t,T]}(\cdot)}\left[ \exp\left(r_f(Y_T)\right)\right]\right) \\
& \approx \overline{\sigma}^2(t) \nabla \left(\mathbb{E}_{Q^y_{[t,T]}(\cdot)}\left[ \ln \left( (f'_\pm)^{-1}\left( \frac{r(Y)}{\alpha}\right) \right)\right] \right).
\end{aligned}
\end{equation}
The optimal initial distribution $\nu_f^*(\cdot)$ is sampled by solving the control problem:
\begin{equation}
q_f^*: = \argmax_q \mathbb{E}_{P^q_{[0,T]}(\cdot)}\left[v_f^*(0, Y_0) - \frac{1}{2} \int_0^T  \left|\frac{q_t}{\sigma'(t)} \right|^2 dt\right],
\end{equation}
with $P^{q_f^*}_{T}(\cdot) = \nu_f^*(\cdot)$.

\section{Numerical experiments}
\label{scnu}

\quad In this section, we apply our stochastic control approach for 
fine-tuning large-scale text-to-image models -- Stable Diffusion v1.5 (SD v1.5).
Specifically, we consider three different types of regularization
as in Section \ref{sc4}: KL divergence, Forward KL divergence, and $\gamma$-divergence. 
The exploration parameter $\alpha$ is evaluated at the levels $\{0.1,1,10\}$,
while the parameter $\gamma$ is tested at the values $\{0.25, 0.5, 1\}$.

\quad We follow the same experimental setup as described in \cite{UZ24}. 
The reward function for fine-tuning is the LAION Aesthetics Predictor V2 model \cite{schuhmann2022laion}, 
which is a linear MLP applied to CLIP embeddings, 
designed to approximate human aesthetic preferences.
All methods are evaluated with respect to three key metrics:
\begin{enumerate}[itemsep = 3 pt]
    \item 
    Reward: the average aesthetic score assigned to the generated images. 
    \item 
    Divergence: the entropy between the sample distribution of the fine-tuned model and that of the pretrained model, which quantifies how much the fine-tuned outputs deviate from those from the baseline model. 
    \item 
    Background ornamentation: the presence of excessive decorative or non-essential visual structures in the background. This phenomenon is a form of reward collapse, 
where the model exploits imperfections in the reward function to artificially increase scores without improving image quality.  
\end{enumerate}

\quad Table \ref{table1} reports the performance of different regularized fine-tuning methods under the three metrics, 
and Figures \ref{fig1}--\ref{fig3} illustrates the outputs of the fine-tuned models with different $\alpha$.
All experiments are conducted on four NVIDIA L40S GPUs.

\quad Our main observations are summarized as follows:
\begin{itemize}
    \item 
    {\em Performance}: 
Forward KL and $\gamma$-divergence consistently achieve higher mean reward values than KL divergence, with Forward KL at $\alpha = 10$ yielding the best overall performance.  
    \item 
    {\em Divergence}: 
Forward KL and $\gamma$-divergence maintain lower divergence from the pretrained output distribution across all settings, 
whereas KL divergence exhibits large drift under small $\alpha$.
    \item 
    {\em Image quality}: 
Forward KL and $\gamma$-divergence substantially reduce background ornamentation compared to KL divergence.
    \item 
    {\em Robustness}: The performance of $\gamma$-divergence is relatively insensitive to changes in $\alpha$, while larger $\alpha$ values improve performance for both Forward KL and $\gamma$-divergence. Moreover, larger $\alpha$ is effective in suppressing background ornamentation. 
\end{itemize}

\begin{table}[htbp!] 
	\centering
	\begin{tabular}{lcc|ccccc} \hline \label{tb:experiment.result}
		Method & $\alpha$ & \textbf{$\gamma$} & Reward & Entropy & Background ornamentation  \\
		\hline\hline
		KL & 10 & & $8.36\pm 0.24 $ & 0.10 & Heavy 
         \\
		Forward &10 &  & $8.95\pm 0.17$ & 0.10 &    \\
		Gamma &10 & 0.25 & $8.57\pm 0.27$ & 0.10 & \\
		Gamma &10 & 0.5 & $8.73\pm 0.28$ & 0.10 &  \\ 
		Gamma &10 & 1 & $8.55\pm 0.39$ & 0.25 & Moderate  \\
        \hline\hline
        KL & 1 &  & $8.30\pm 0.22$ & 0.31 &  \\
		Forward & 1 & & $8.50\pm 0.23$  & 0.09 & Heavy  \\
		Gamma & 1 & 1 & $8.49\pm 0.20$ & 0.11 & Moderate  \\
        \hline\hline
        KL & 0.1 &  & $8.13\pm 0.31 $ & 0.57 & Moderate   \\
		Forward &0.1 &  & $8.45\pm 0.23$ & 0.10 &  \\
		Gamma &0.1 & 0.25 & $8.30\pm$ 0.28 & 0.27 & Heavy  \\
		Gamma &0.1 & 0.5 & $8.73\pm 0.23$ & 0.17 & Moderate  \\
		Gamma &0.1 & 1 & $8.65\pm 0.24$  & 0.10 & Moderate   \\\hline \hline
	\end{tabular}
    \caption{Evaluation of fine-tuning SD v1.5.}
    \label{table1}
\end{table}

\begin{figure*}[ht]
	\centering
	\begin{adjustbox}{width=.78\textwidth}
		\begin{tabular}{@{}l*{8}{c@{}} }
			& cat & dog & horse & monkey  & bird & tiger & lizard \\
			\multirow{2}{*}{\stackanchor{KL}{$\alpha=10$}} 
& \includegraphics[width=0.13\linewidth]{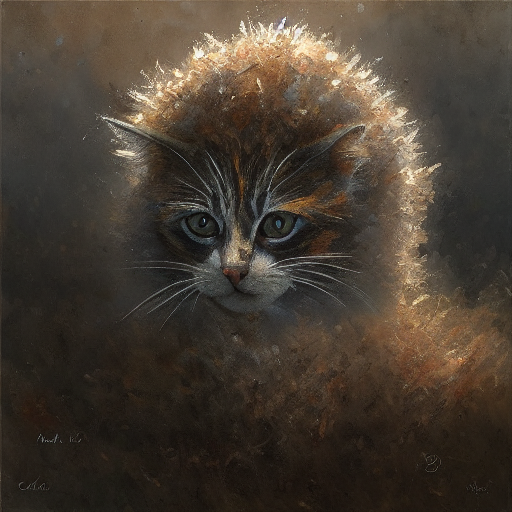}
& \includegraphics[width=0.13\linewidth]{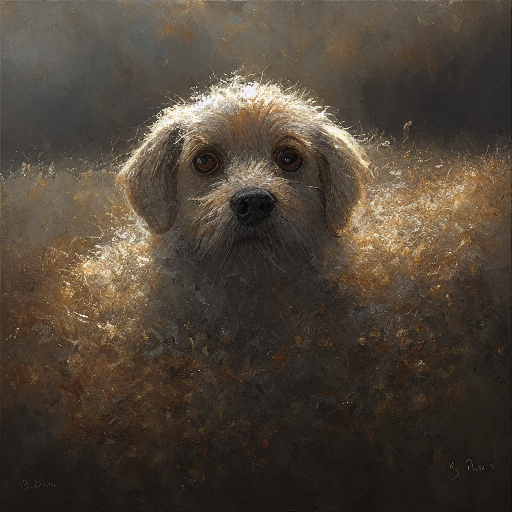}
& \includegraphics[width=0.13\linewidth]{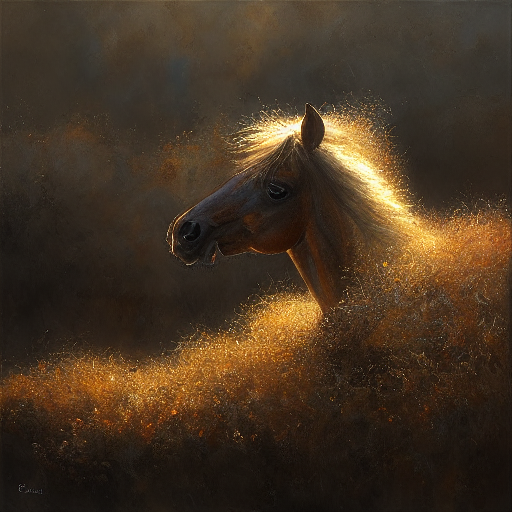}
& \includegraphics[width=0.13\linewidth]{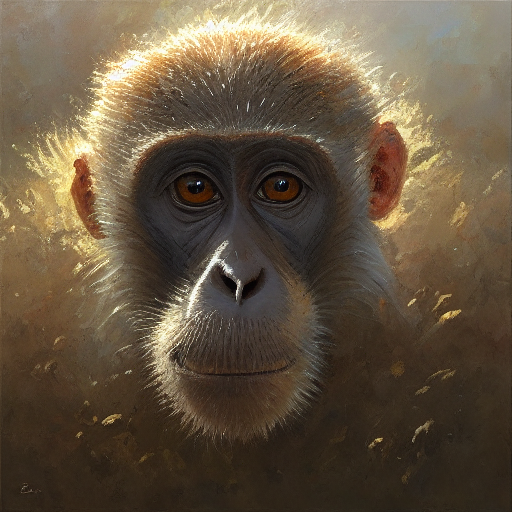}
& \includegraphics[width=0.13\linewidth]{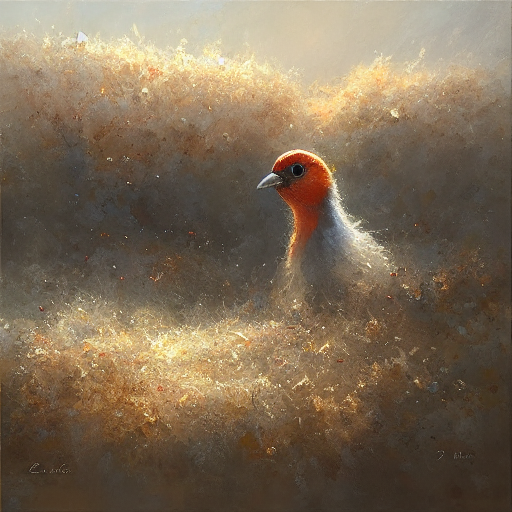}
& \includegraphics[width=0.13\linewidth]{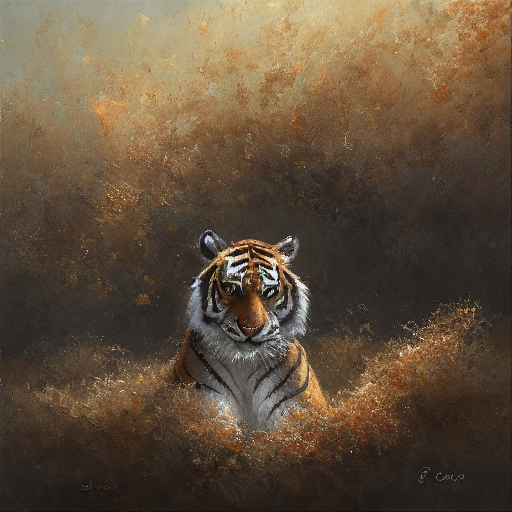} 
& \includegraphics[width=0.13\linewidth]{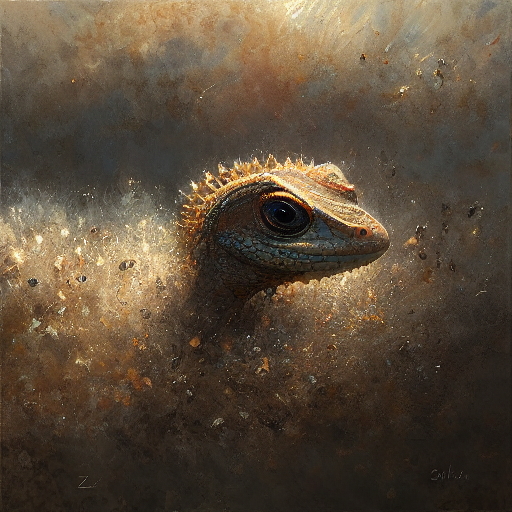} \\
& 8.23 & 8.47 & 7.98 & 8.38  & 8.47 & 8.13 & 8.87  \\
			\multirow{2}{*}{\stackanchor{Forward}{$\alpha=10$}} 
& \includegraphics[width=0.13\linewidth]{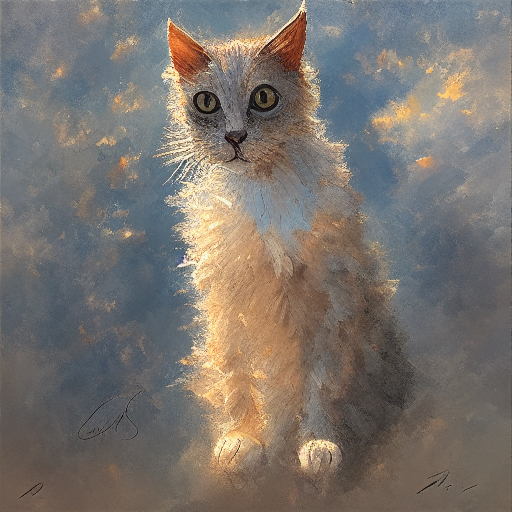}
& \includegraphics[width=0.13\linewidth]{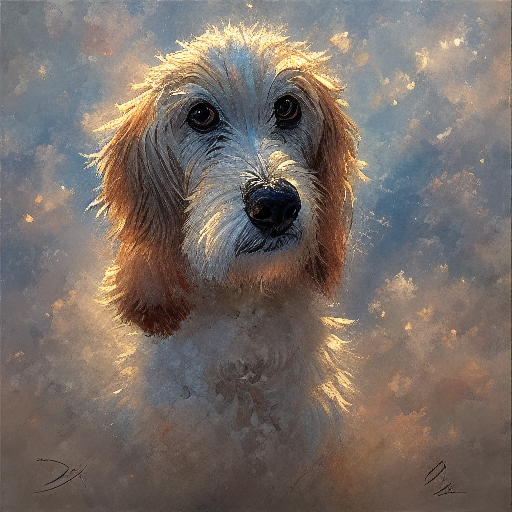}
& \includegraphics[width=0.13\linewidth]{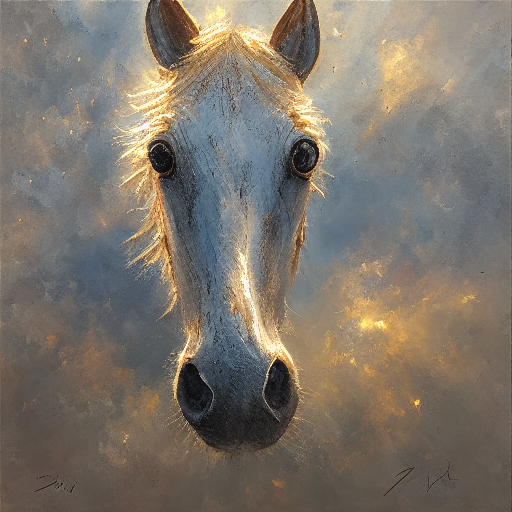}
& \includegraphics[width=0.13\linewidth]{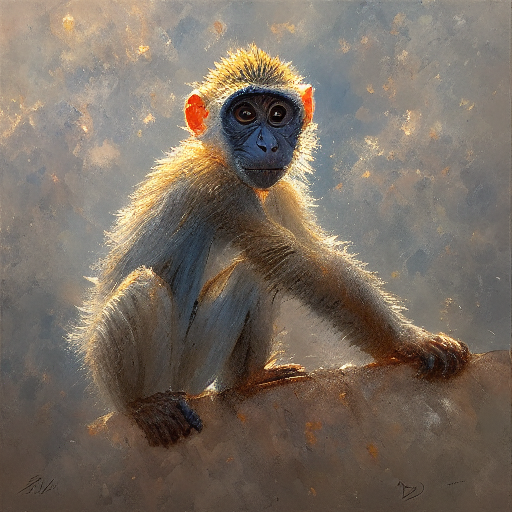}
& \includegraphics[width=0.13\linewidth]{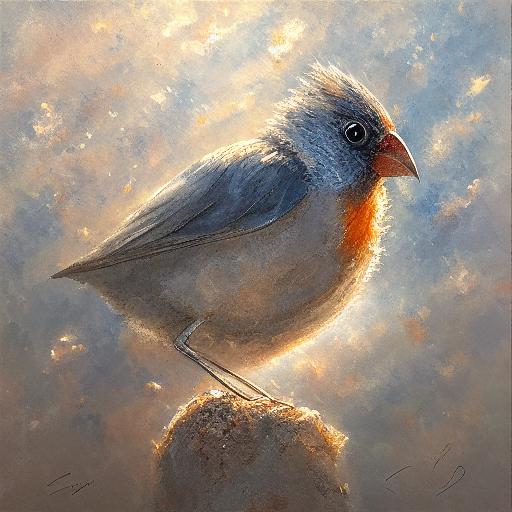}
& \includegraphics[width=0.13\linewidth]{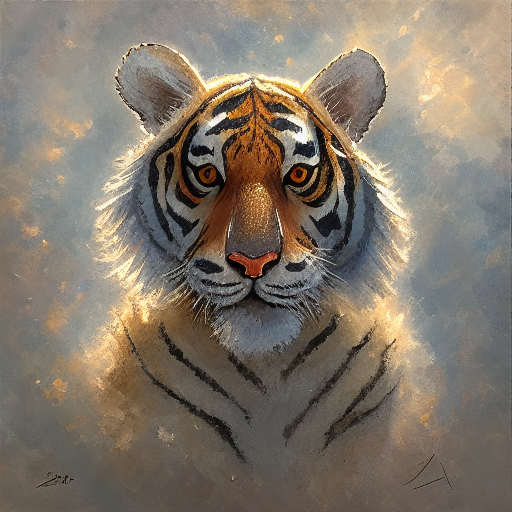} 
& \includegraphics[width=0.13\linewidth]{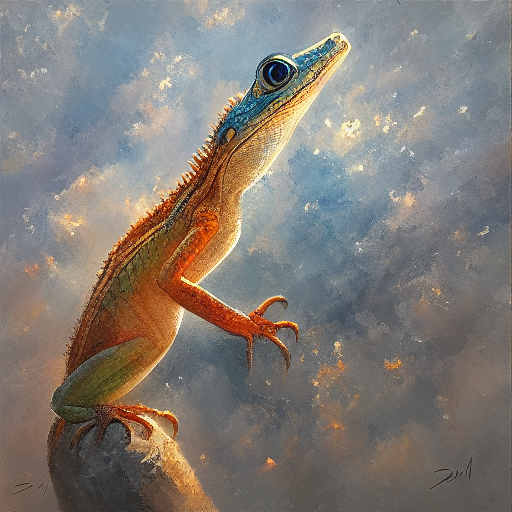} \\
& 8.73 & 9.00 & 8.77 & 9.01  &8.94 & 8.74 & 8.93 \\
\multirow{2}{*}{\stackanchor{Gamma}{\stackanchor{$\alpha=10$}{$\gamma=0.25$}}} 
& \includegraphics[width=0.13\linewidth]{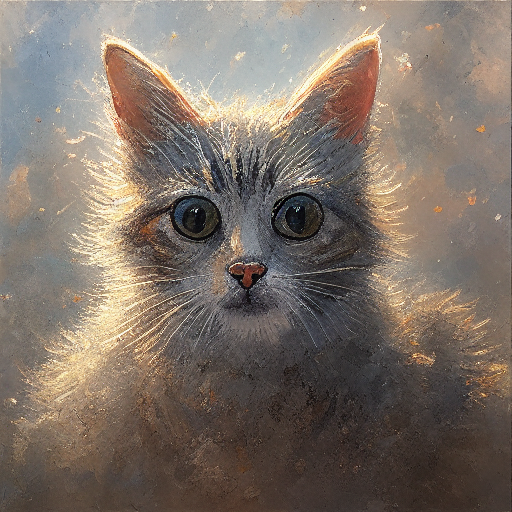}
& \includegraphics[width=0.13\linewidth]{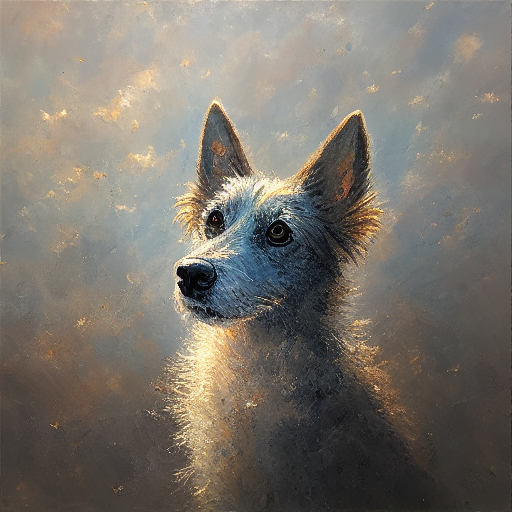}
& \includegraphics[width=0.13\linewidth]{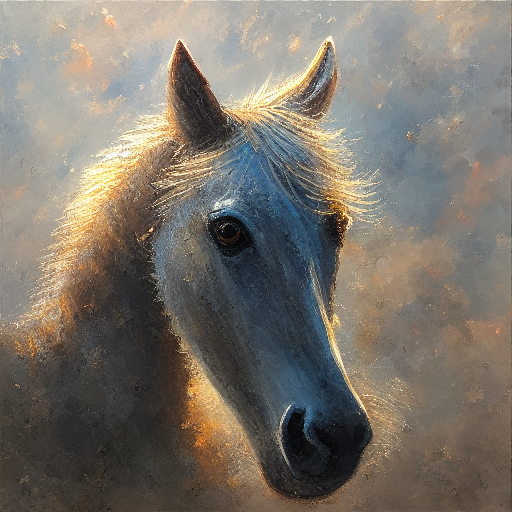}
& \includegraphics[width=0.13\linewidth]{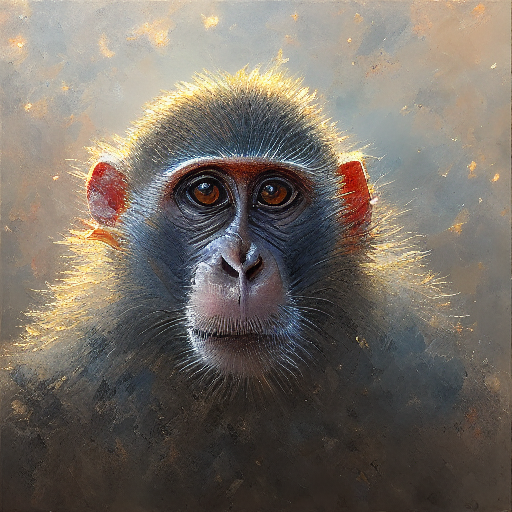}
& \includegraphics[width=0.13\linewidth]{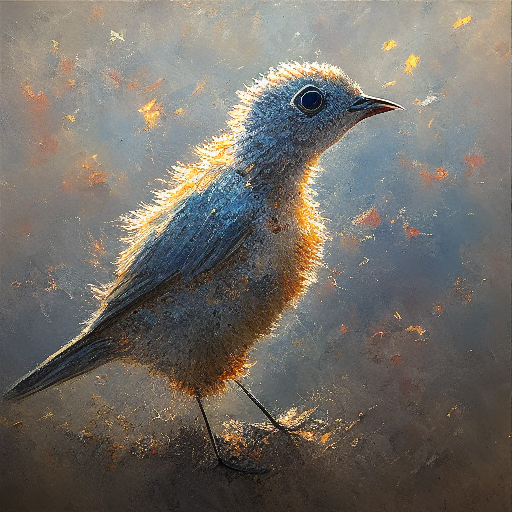}
& \includegraphics[width=0.13\linewidth]{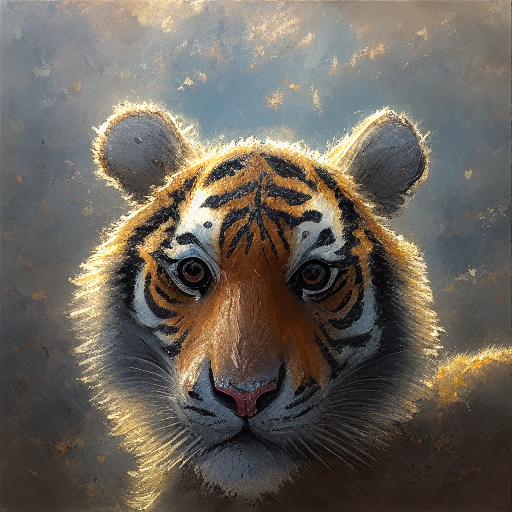} 
& \includegraphics[width=0.13\linewidth]{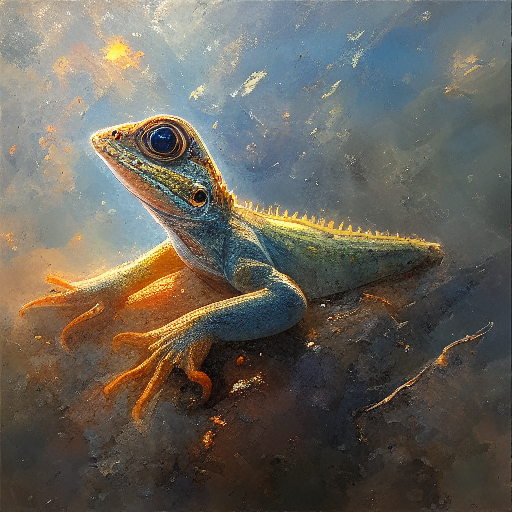} \\
& 8.49 & 8.84 & 8.35 & 8.47  & 8.62 & 8.76 & 8.53 \\
\multirow{2}{*}{\stackanchor{Gamma}{\stackanchor{$\alpha=10$}{$\gamma=0.5$}}} 
& \includegraphics[width=0.13\linewidth]{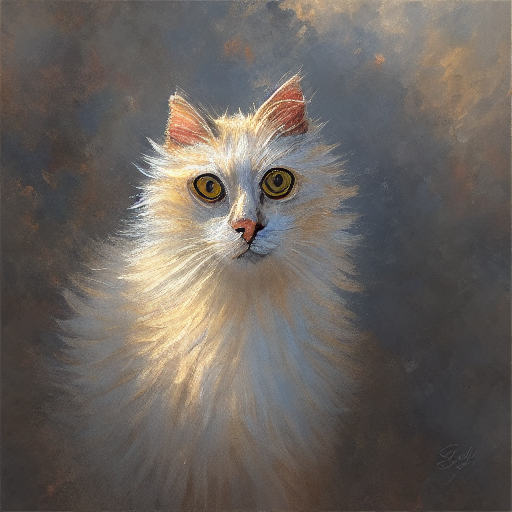}
& \includegraphics[width=0.13\linewidth]{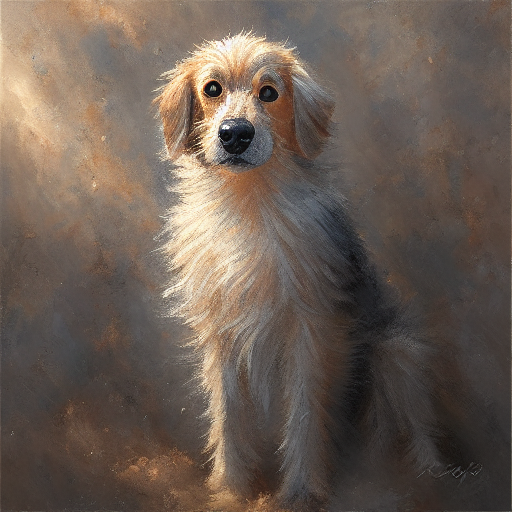}
& \includegraphics[width=0.13\linewidth]{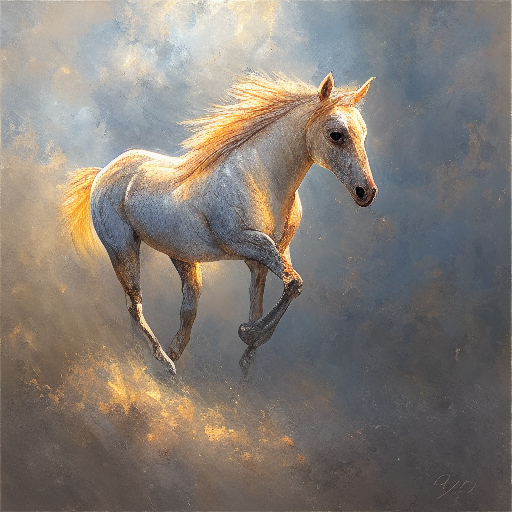}
& \includegraphics[width=0.13\linewidth]{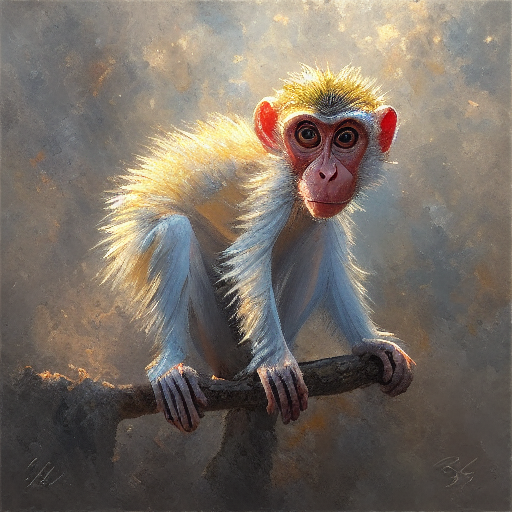}
& \includegraphics[width=0.13\linewidth]{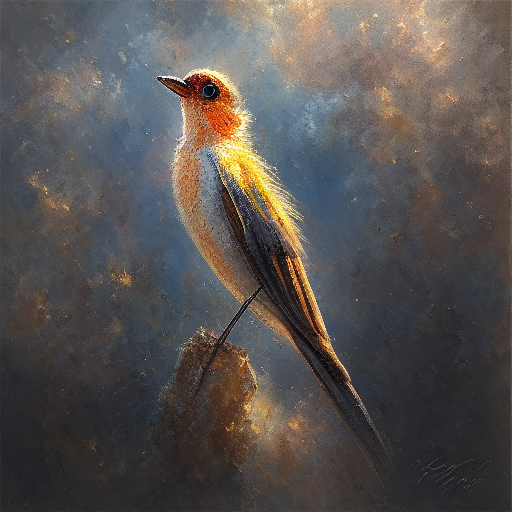}
& \includegraphics[width=0.13\linewidth]{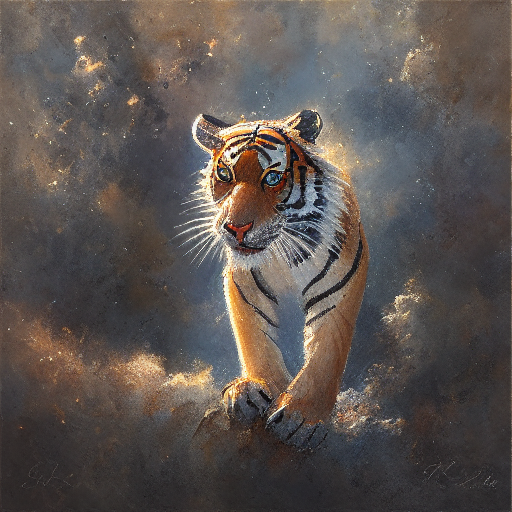} 
& \includegraphics[width=0.13\linewidth]{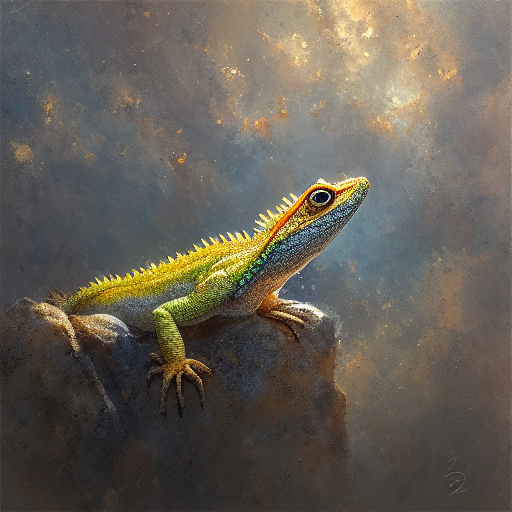} \\
& 8.72 & 8.50 & 8.00 & 8.97  & 8.64 & 8.67 & 8.74 \\
\multirow{2}{*}{\stackanchor{Gamma}{\stackanchor{$\alpha=10$}{$\gamma=1$}}} 
& \includegraphics[width=0.13\linewidth]{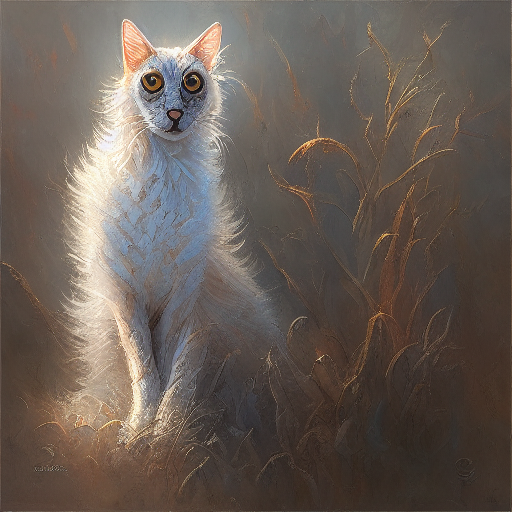}
& \includegraphics[width=0.13\linewidth]{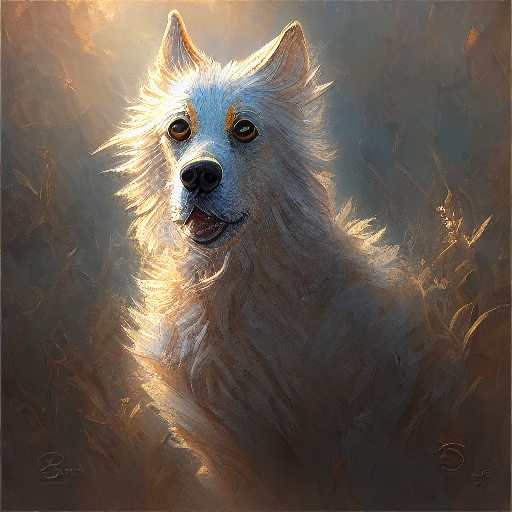}
& \includegraphics[width=0.13\linewidth]{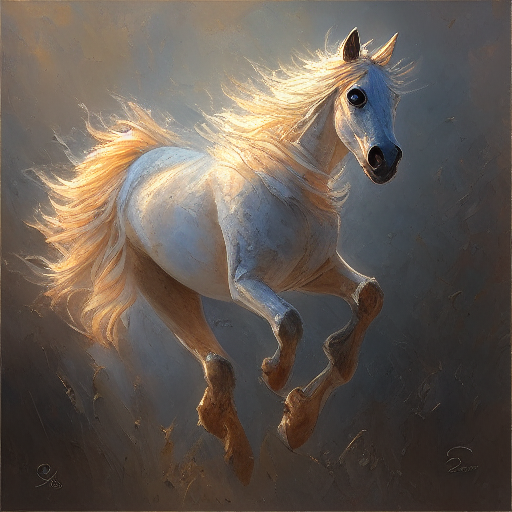}
& \includegraphics[width=0.13\linewidth]{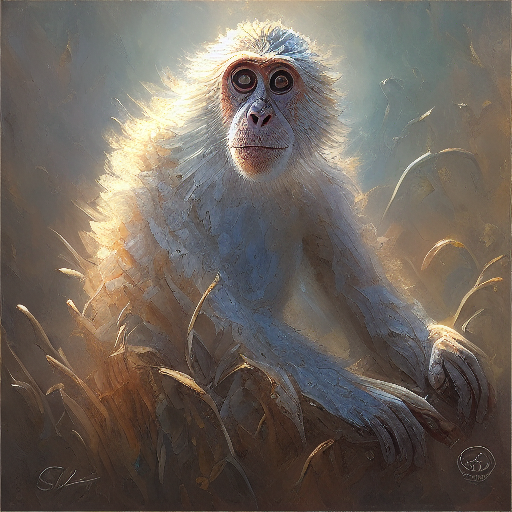}
& \includegraphics[width=0.13\linewidth]{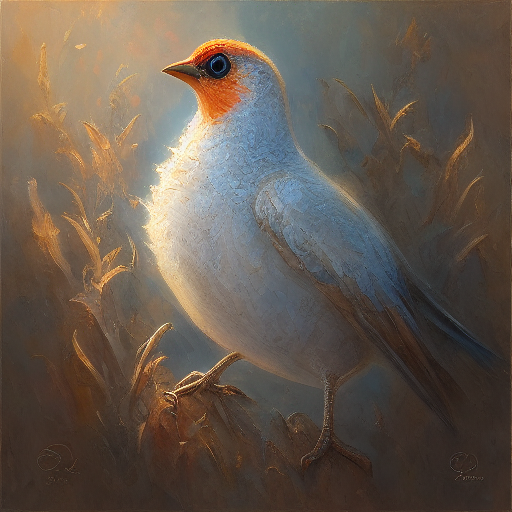}
& \includegraphics[width=0.13\linewidth]{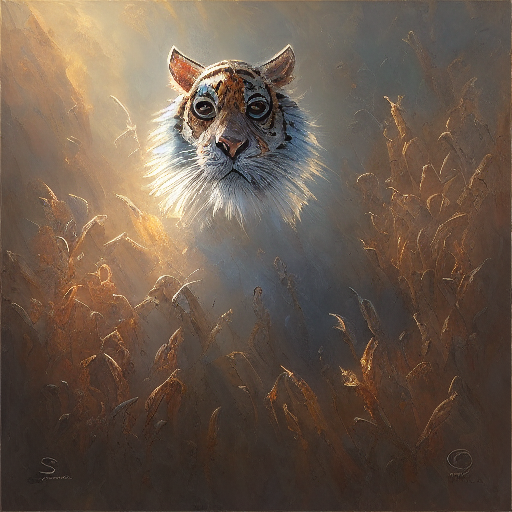} 
& \includegraphics[width=0.13\linewidth]{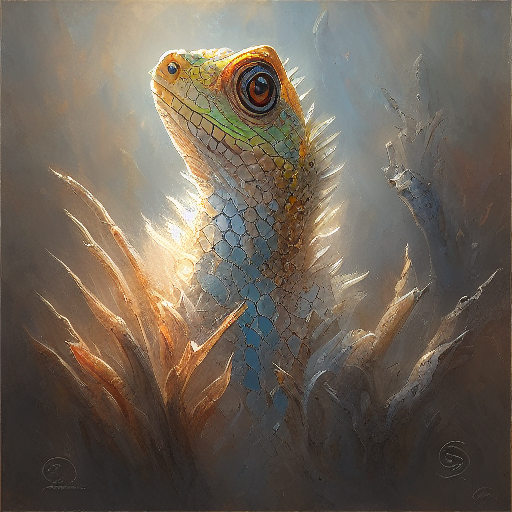} \\
& 8.47 & 8.75 & 8.47 & 8.71 & 8.54 & 8.59 & 8.53 & 8.69 \\
		\end{tabular}
	\end{adjustbox}
	\caption{Generated images of the fine-tuned models with $\alpha=10$.}
    \label{fig1}
\end{figure*}

\begin{figure*}[ht]
	\centering
	\begin{adjustbox}{width=.78\textwidth}
		\begin{tabular}{@{}l*{8}{c@{}} }
			& cat & dog & horse & monkey &  bird & tiger & lizard \\
	\multirow{2}{*}{\stackanchor{KL}{$\alpha=1$}} 
	& \includegraphics[width=0.13\linewidth]{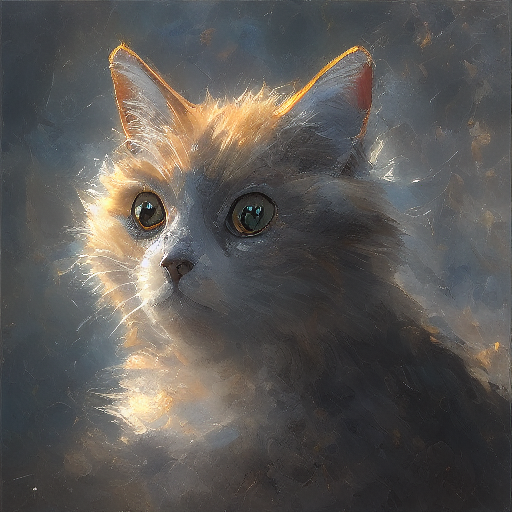}
	& \includegraphics[width=0.13\linewidth]{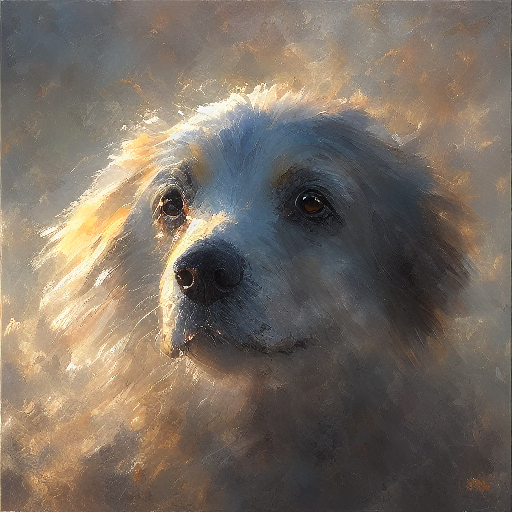}
	& \includegraphics[width=0.13\linewidth]{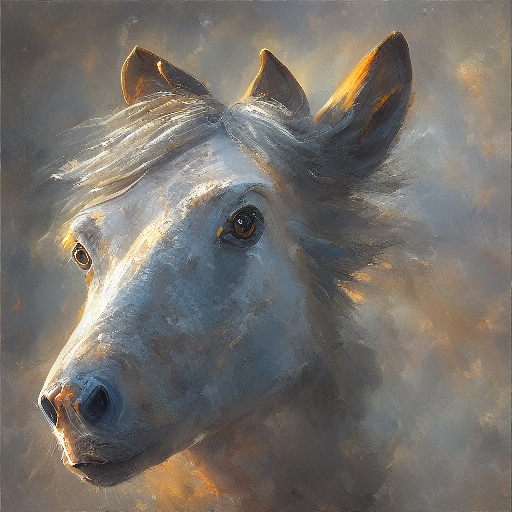}
	& \includegraphics[width=0.13\linewidth]{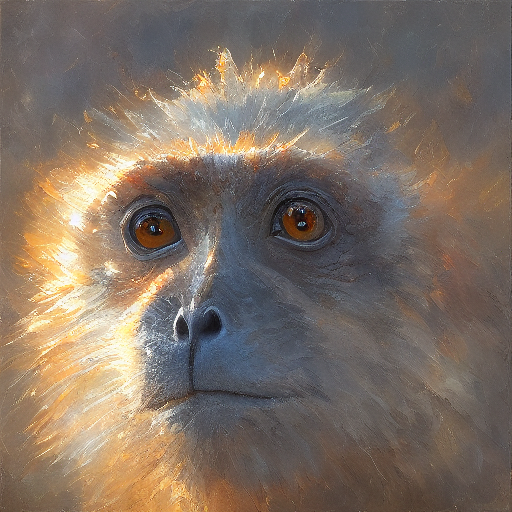}
	& \includegraphics[width=0.13\linewidth]{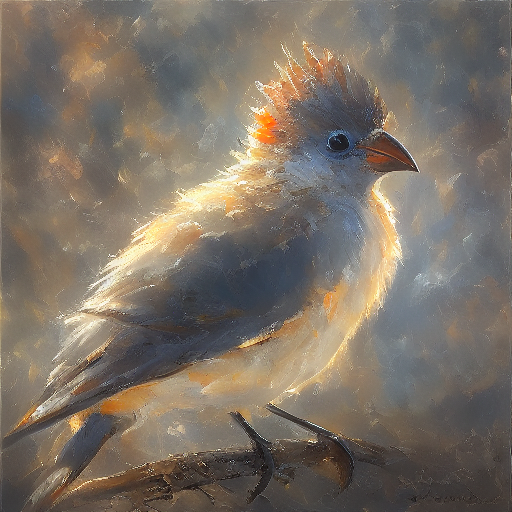}
	& \includegraphics[width=0.13\linewidth]{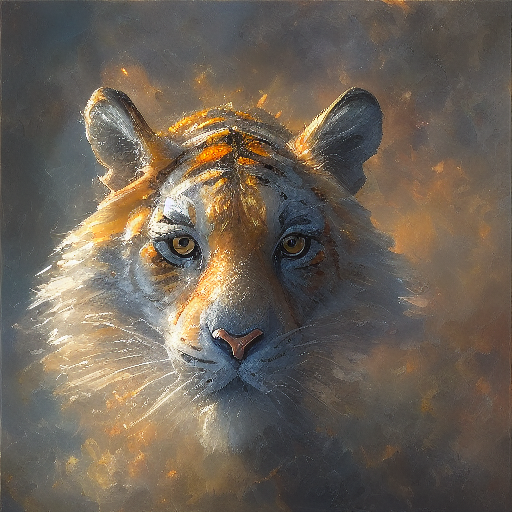} 
	& \includegraphics[width=0.13\linewidth]{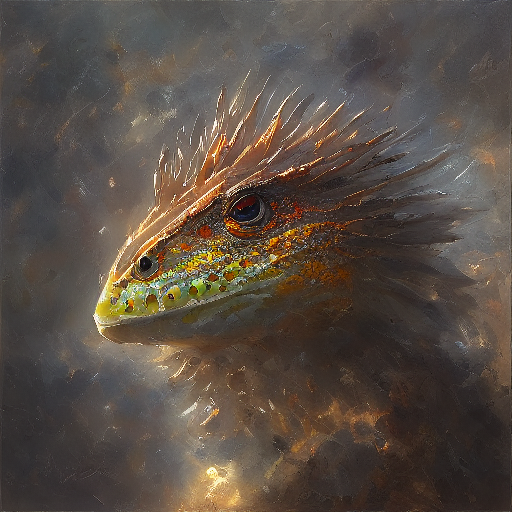} \\
	& 8.17 & 8.18 & 8.21 & 8.24  & 8.38 & 8.38 & 8.24  \\
	\multirow{2}{*}{\stackanchor{Forward}{$\alpha=1$}} 
	& \includegraphics[width=0.13\linewidth]{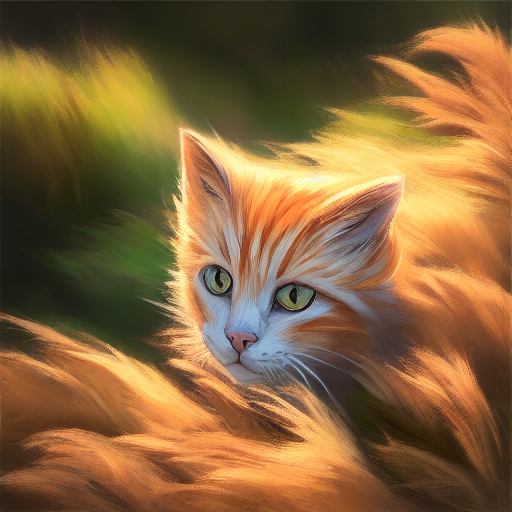}
	& \includegraphics[width=0.13\linewidth]{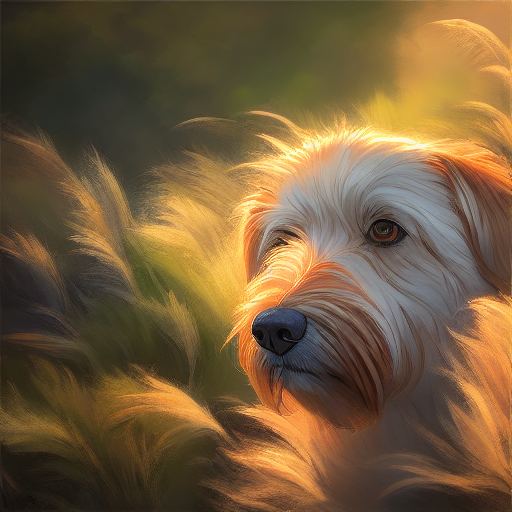}
	& \includegraphics[width=0.13\linewidth]{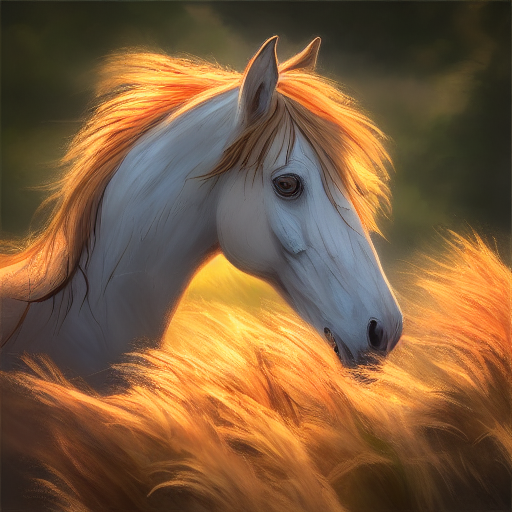}
	& \includegraphics[width=0.13\linewidth]{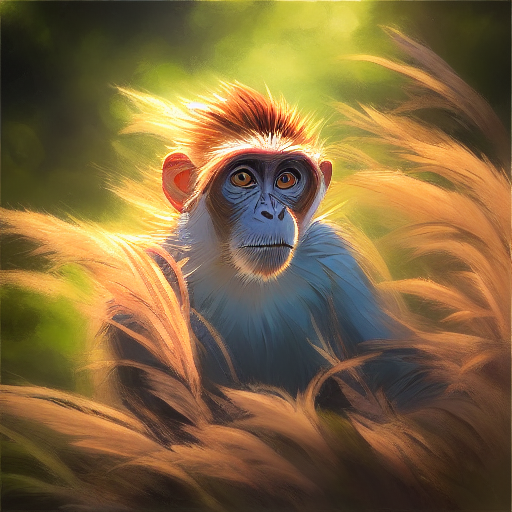}
	& \includegraphics[width=0.13\linewidth]{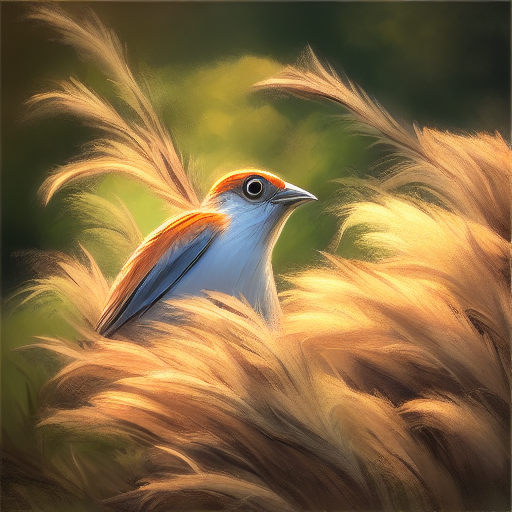}
	& \includegraphics[width=0.13\linewidth]{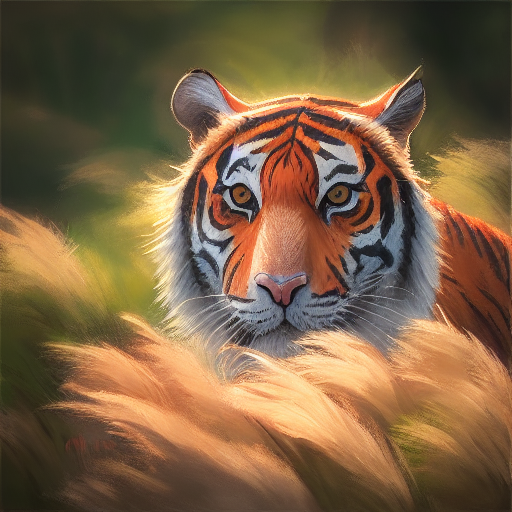} 
	& \includegraphics[width=0.13\linewidth]{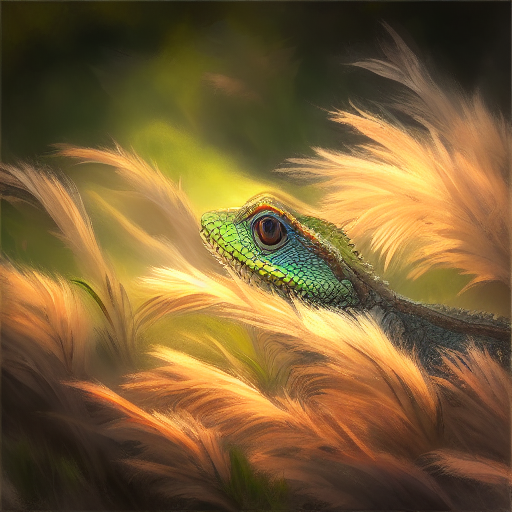} \\
	& 8.32 & 8.58 & 8.08 & 8.62  &8.64 & 8.73 & 8.70 \\
	\multirow{2}{*}{\stackanchor{Gamma}{\stackanchor{$\alpha=1$}{$\gamma=1$}}} 
	& \includegraphics[width=0.13\linewidth]{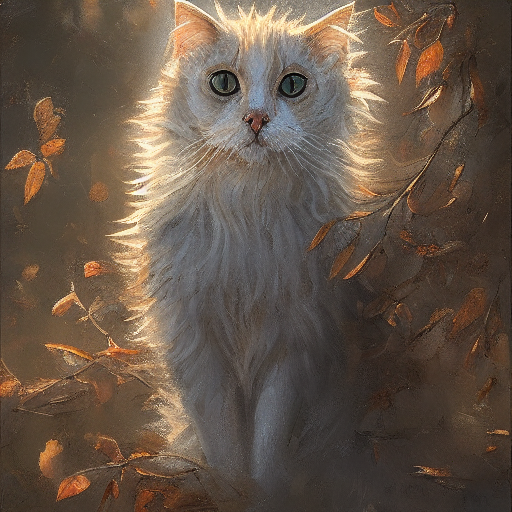}
	& \includegraphics[width=0.13\linewidth]{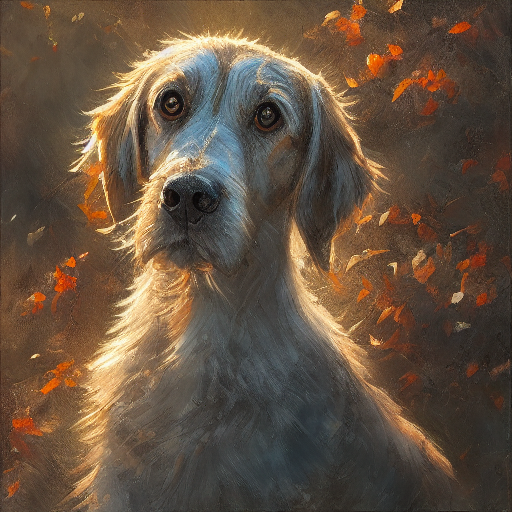}
	& \includegraphics[width=0.13\linewidth]{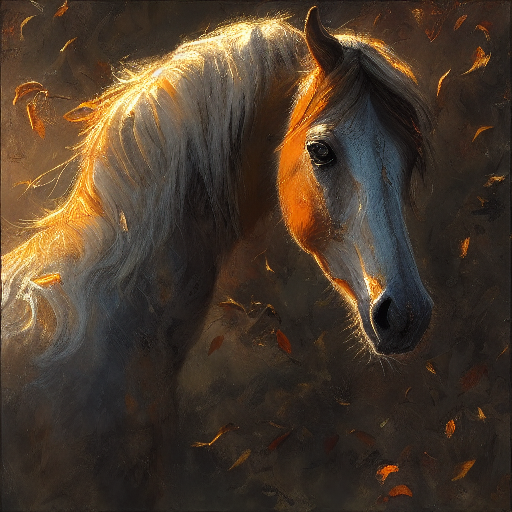}
	& \includegraphics[width=0.13\linewidth]{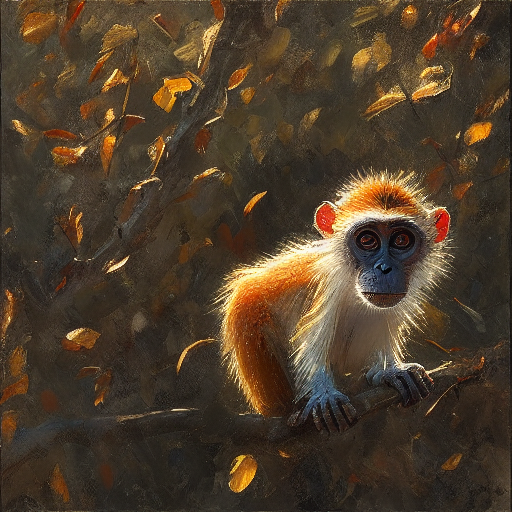}
	& \includegraphics[width=0.13\linewidth]{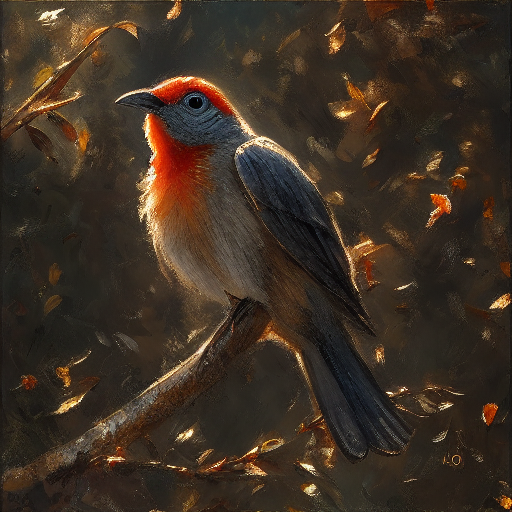}
	& \includegraphics[width=0.13\linewidth]{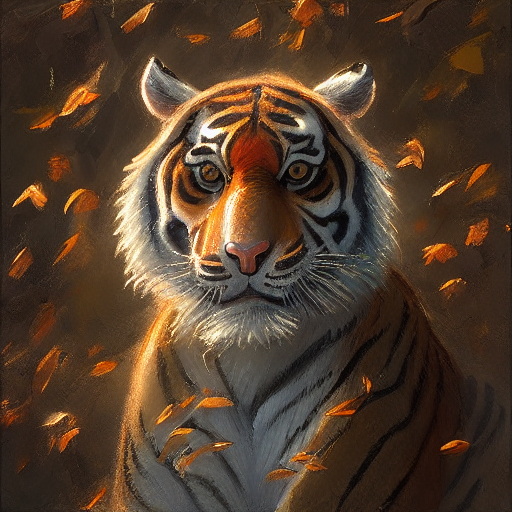} 
	& \includegraphics[width=0.13\linewidth]{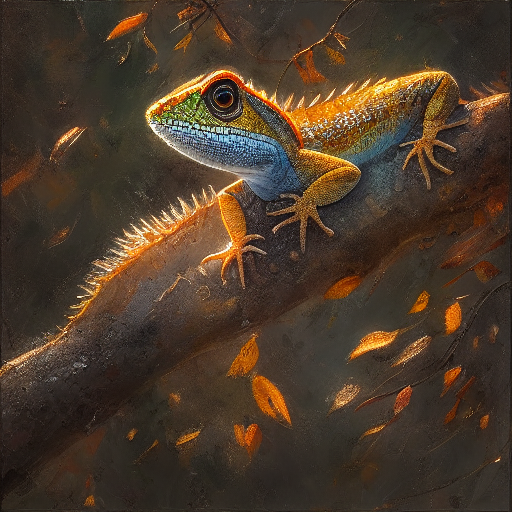} \\
	& 8.47 & 8.77 & 8.35 & 8.62 & 8.69 & 8.66 & 8.66 \\
		\end{tabular}
	\end{adjustbox}
	\caption{Generated images of the fine-tuned models with $\alpha=1$.}
        \label{fig2}
\end{figure*}

\begin{figure*}[ht]
	\centering
	\begin{adjustbox}{width=.78\textwidth}
		\begin{tabular}{@{}l*{8}{c@{}} }
& cat & dog & horse & monkey &  bird & tiger & lizard \\
\multirow{2}{*}{\stackanchor{KL}{$\alpha=0.1$}} 
& \includegraphics[width=0.13\linewidth]{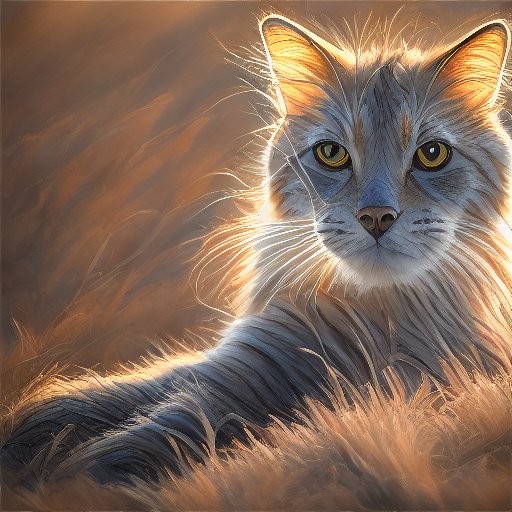}
& \includegraphics[width=0.13\linewidth]{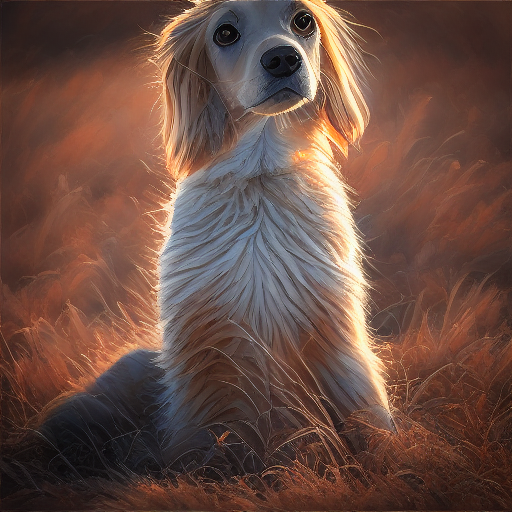}
& \includegraphics[width=0.13\linewidth]{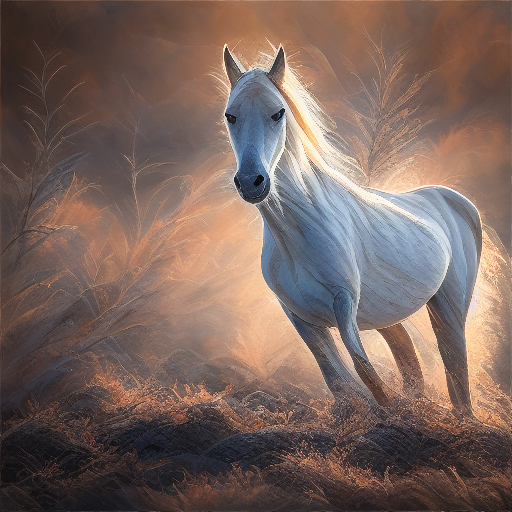}
& \includegraphics[width=0.13\linewidth]{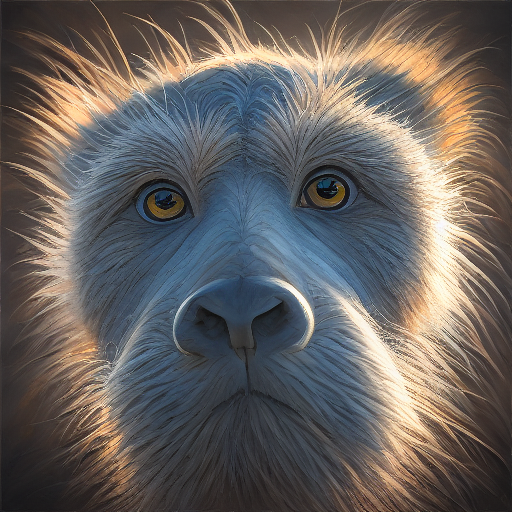}
& \includegraphics[width=0.13\linewidth]{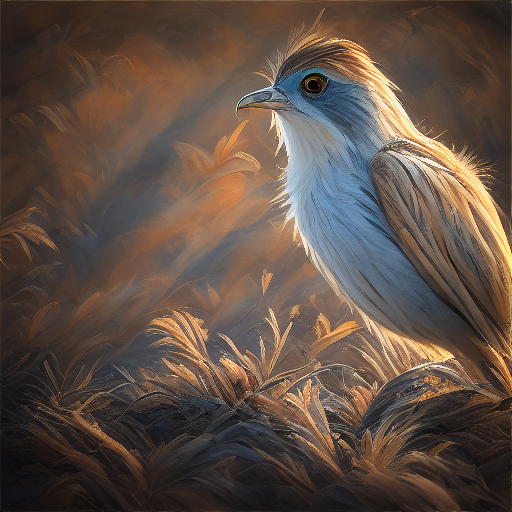}
& \includegraphics[width=0.13\linewidth]{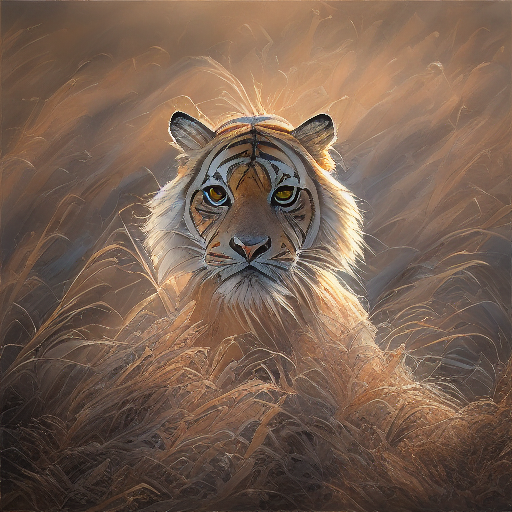} 
& \includegraphics[width=0.13\linewidth]{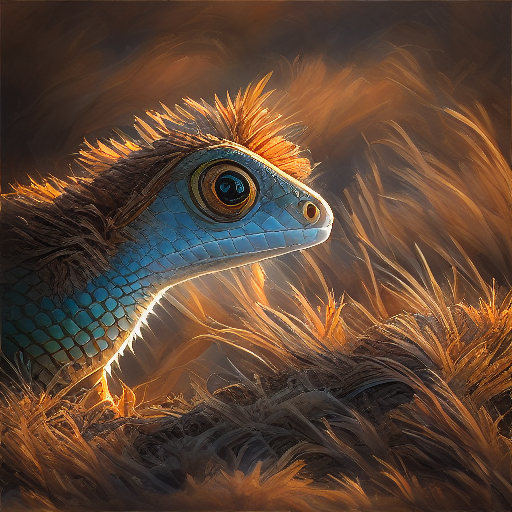} \\
& 8.01 & 8.18 & 7.81 & 8.14 & 8.26 & 8.24 & 8.46  \\
\multirow{2}{*}{\stackanchor{Forward}{$\alpha=0.1$}} 
& \includegraphics[width=0.13\linewidth]{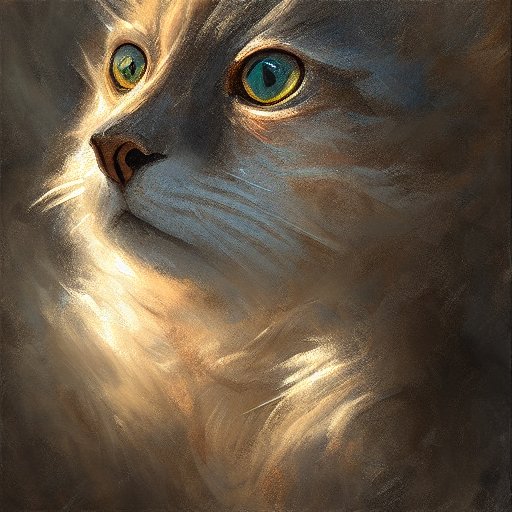}
& \includegraphics[width=0.13\linewidth]{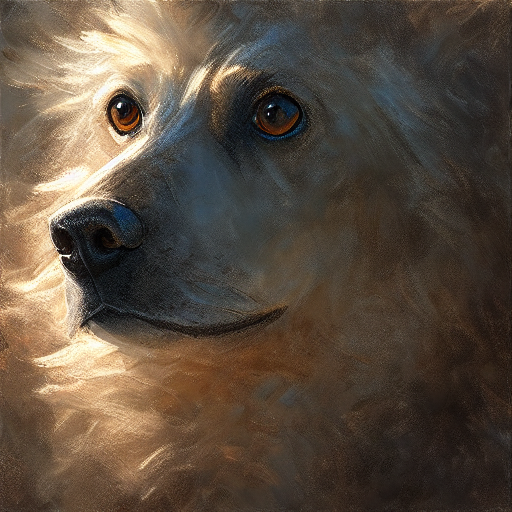}
& \includegraphics[width=0.13\linewidth]{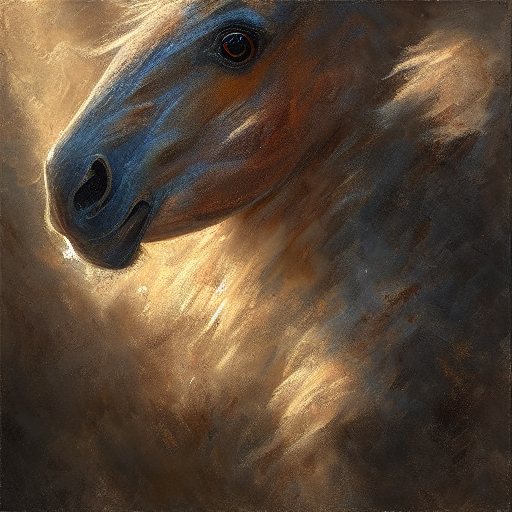}
& \includegraphics[width=0.13\linewidth]{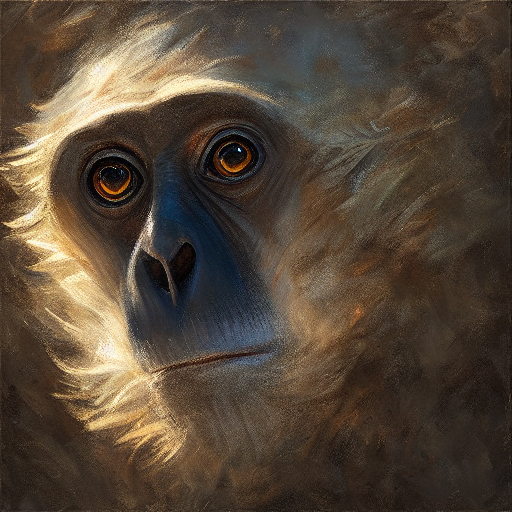}
& \includegraphics[width=0.13\linewidth]{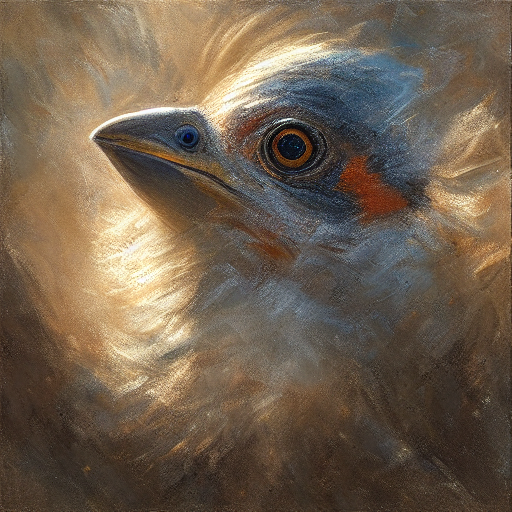}
& \includegraphics[width=0.13\linewidth]{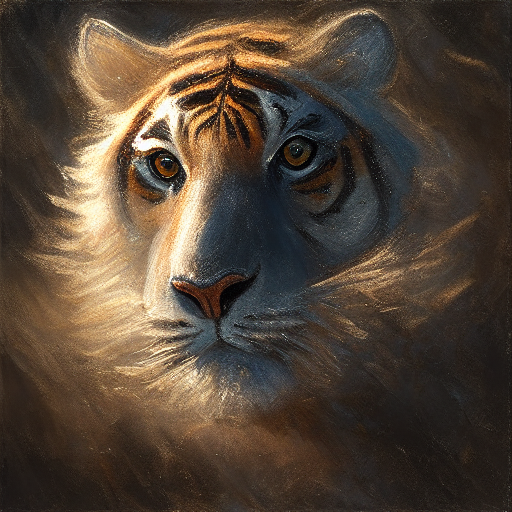} 
& \includegraphics[width=0.13\linewidth]{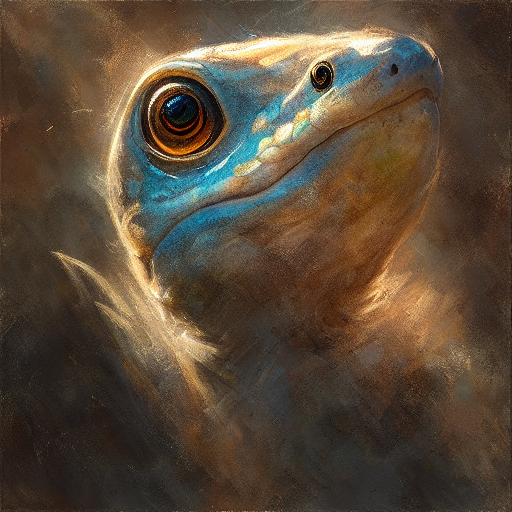} \\
& 8.43 & 8.58 & 8.09 & 8.86  & 8.31 & 8.43 & 8.52 \\
\multirow{2}{*}{\stackanchor{Gamma}{\stackanchor{$\alpha=0.1$}{$\gamma=0.25$}}} 
& \includegraphics[width=0.13\linewidth]{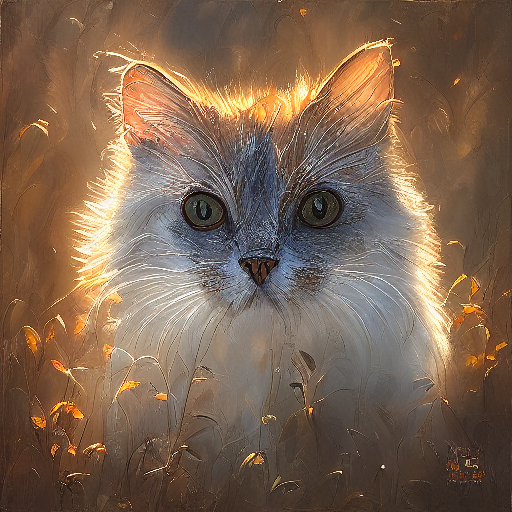}
& \includegraphics[width=0.13\linewidth]{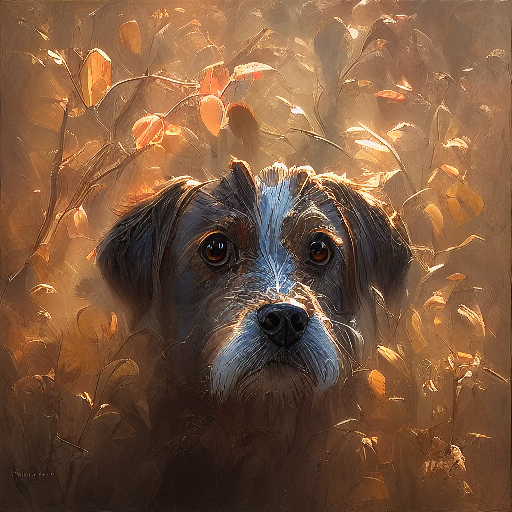}
& \includegraphics[width=0.13\linewidth]{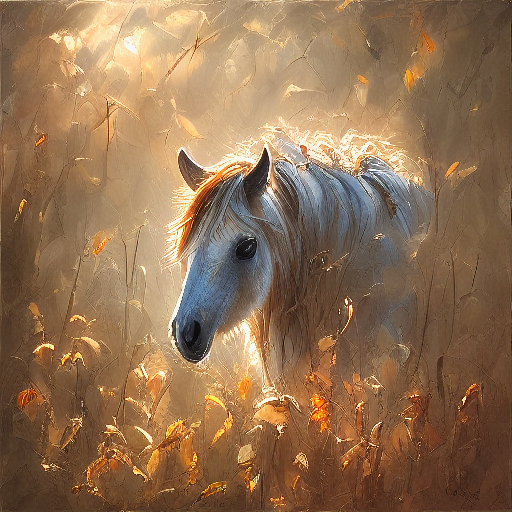}
& \includegraphics[width=0.13\linewidth]{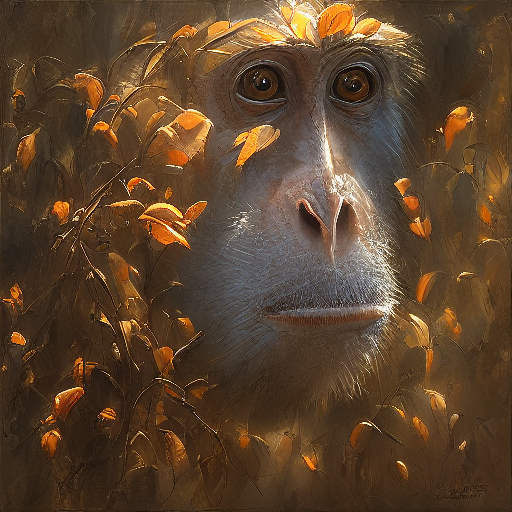}
& \includegraphics[width=0.13\linewidth]{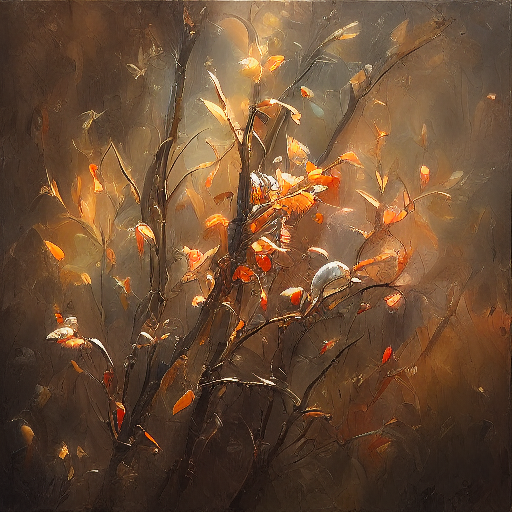}
& \includegraphics[width=0.13\linewidth]{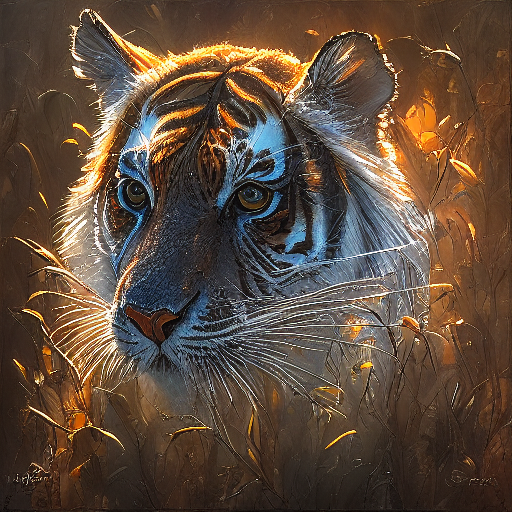} 
& \includegraphics[width=0.13\linewidth]{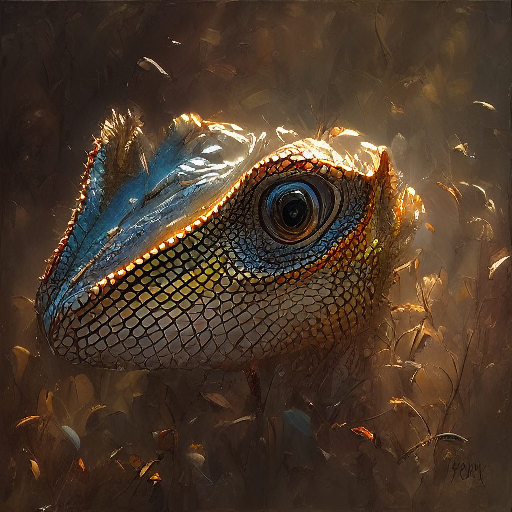} \\
&8.37& 8.37& 8.18& 8.70&  7.99& 8.28& 8.22  \\
\multirow{2}{*}{\stackanchor{Gamma}{\stackanchor{$\alpha=0.1$}{$\gamma=0.5$}}} 
& \includegraphics[width=0.13\linewidth]{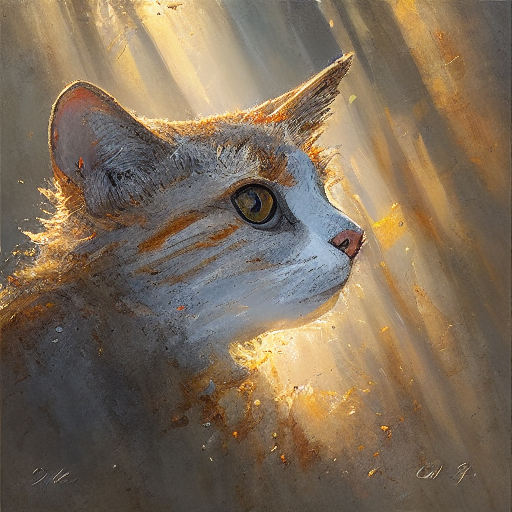}
& \includegraphics[width=0.13\linewidth]{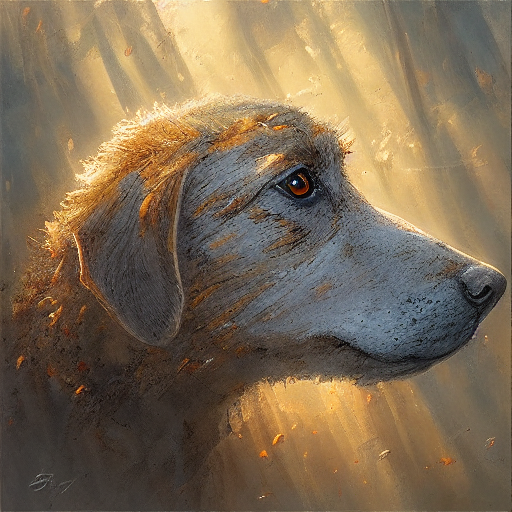}
& \includegraphics[width=0.13\linewidth]{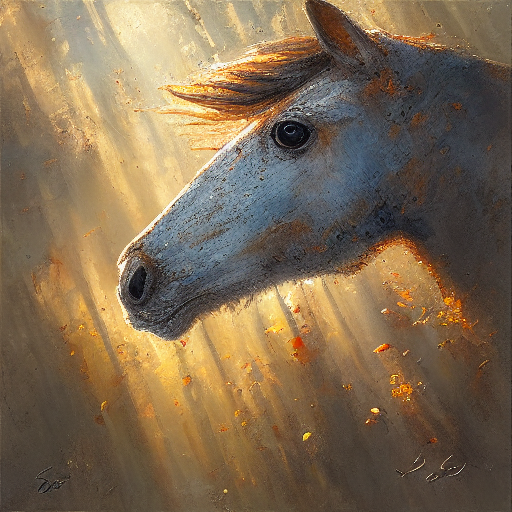}
& \includegraphics[width=0.13\linewidth]{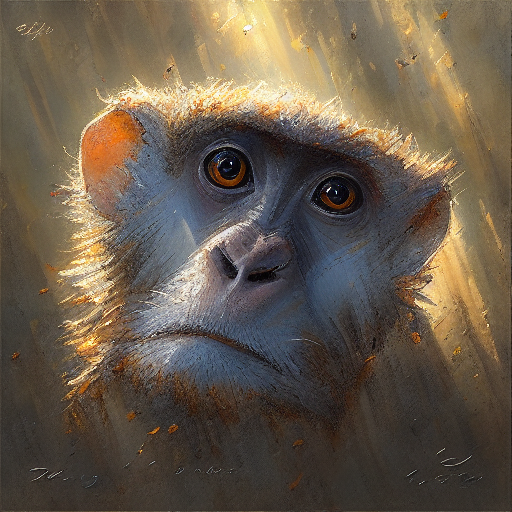}
& \includegraphics[width=0.13\linewidth]{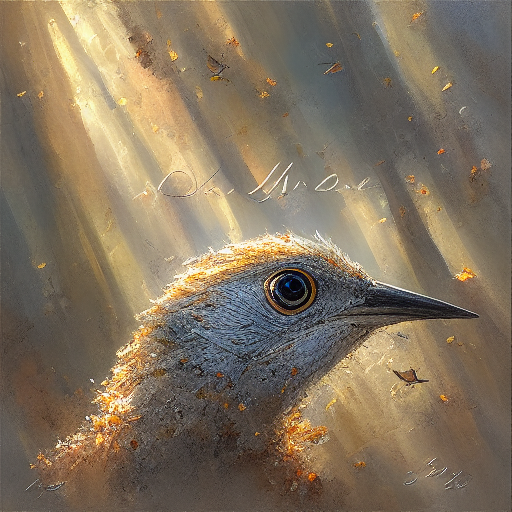}
& \includegraphics[width=0.13\linewidth]{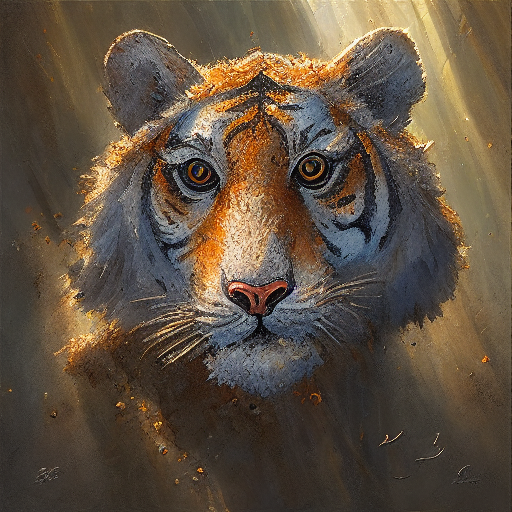} 
& \includegraphics[width=0.13\linewidth]{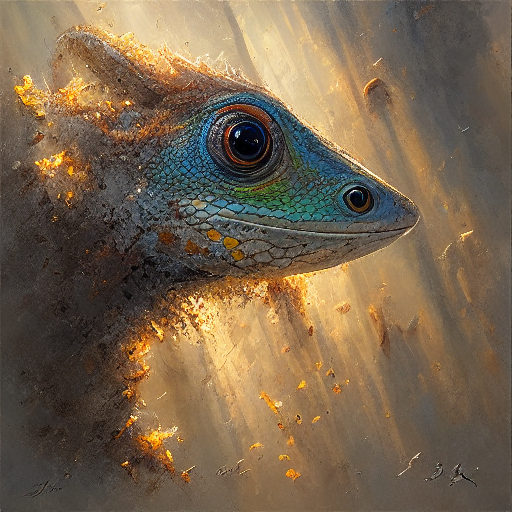} \\
& 8.72 & 8.64 & 8.64 & 8.74  & 8.67 & 8.77 & 8.90 \\
\multirow{2}{*}{\stackanchor{Gamma}{\stackanchor{$\alpha=0.1$}{$\gamma=1$}}} 
& \includegraphics[width=0.13\linewidth]{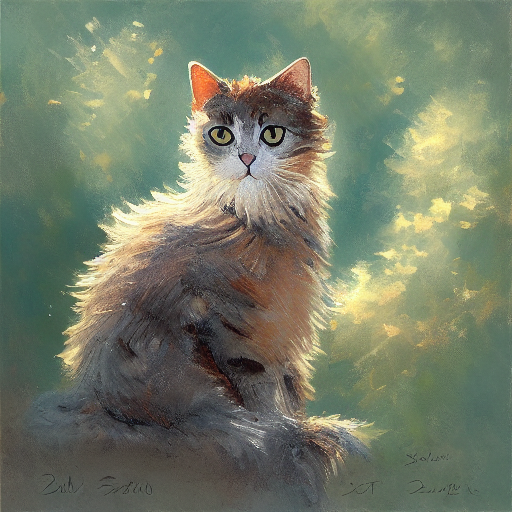}
& \includegraphics[width=0.13\linewidth]{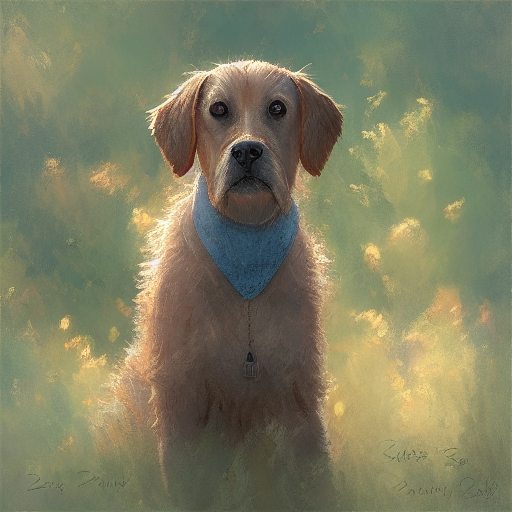}
& \includegraphics[width=0.13\linewidth]{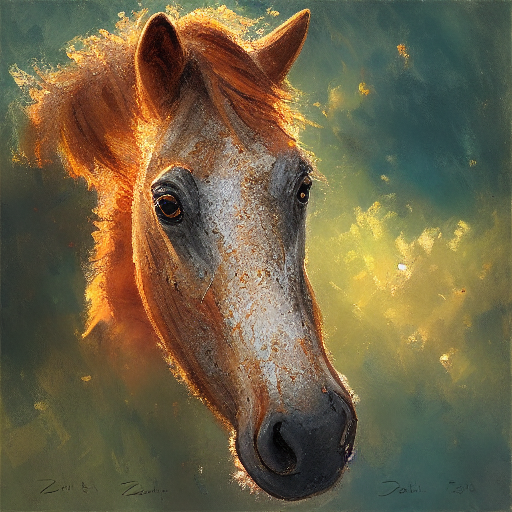}
& \includegraphics[width=0.13\linewidth]{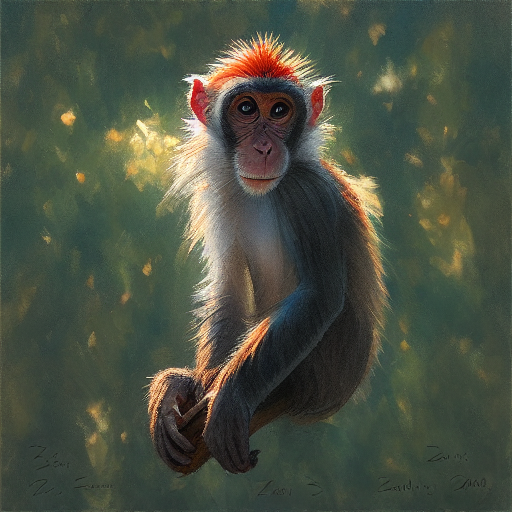}
& \includegraphics[width=0.13\linewidth]{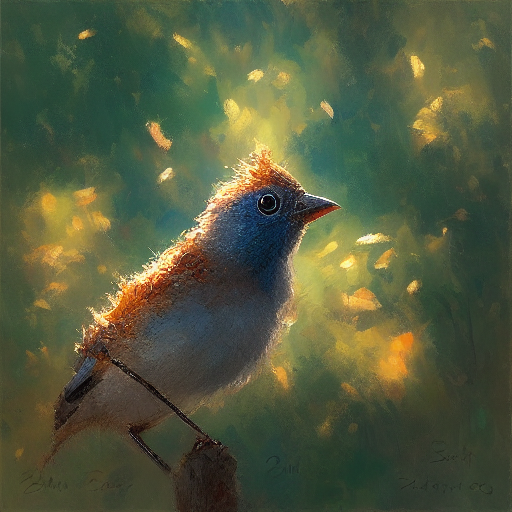}
& \includegraphics[width=0.13\linewidth]{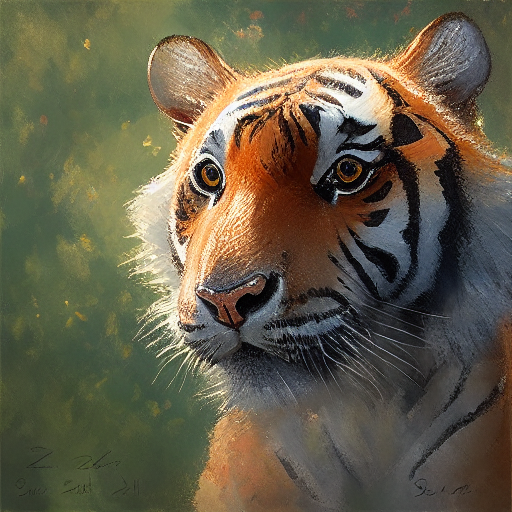} 
& \includegraphics[width=0.13\linewidth]{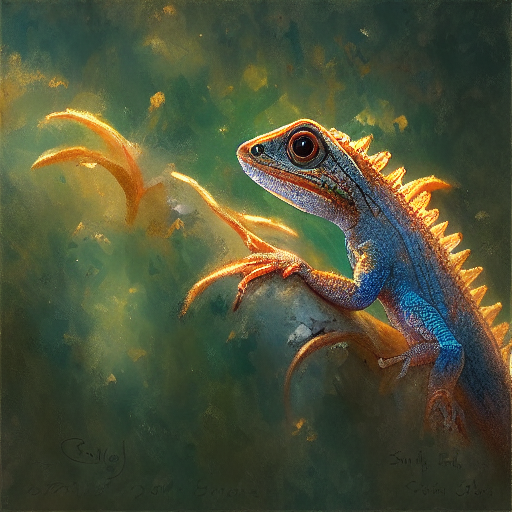} \\
& 8.48 & 8,67 & 8.58 & 8.70  & 8.65 & 8.69 & 8.52  \\
		\end{tabular}
	\end{adjustbox}
	\caption{Generated images of the fine-tuned models with $\alpha=0.1$.}
        \label{fig3}
\end{figure*}

\section{Conclusion}
\label{sc5}

\quad In this paper, 
we give a rigorous treatment to
the theory of fine-tuning by entropy-regularized stochastic control,
which was recently proposed by \cite{UZ24} in the context of 
continuous-time diffusion models. 
We also generalize to the setting where an $f$-divergence regularizer is introduced.
Numerical results are conducted on large scale text-to-image models, 
which suggest using Forward KL or $\gamma$-divergence for regularization,
with a large exploration parameter $\alpha$.

\quad There are several directions to extend this work. 
First, we provide an entropy-regularized approach to emulate the fine-tuned distribution \eqref{eq:fdivft} 
regularized by $f$-divergence.
It is interesting to know whether it can be sampled by directly 
regularizing the controlled process by $f$-divergence. 
Also it remains unknown whether the fine-tuned distribution \eqref{eq:entfFT}
can be generated by solving some stochastic control problem.
Finally, the fine-tuned distribution \eqref{eq:fdivft} and the stochastic control approach \eqref{eq:scf}
can be applied to other real data such as protein sequence generation.

\bigskip
{\bf Acknowledgement:} 
We thank Haoxian Chen, Minshuo Chen and Hanyang Zhao for various pointers to the literature.
We thank David Yao and Xun Yu Zhou for stimulating discussions.
W. Tang gratefully acknowledges financial support through NSF grant DMS-2206038,
the Columbia Innovation Hub grant, and the Tang Family Assistant Professorship.
F. Zhou acknowledges the support from a Columbia-CityU/HK collaborative project led by InnoHK Initiative, The
Government of the HKSAR and the AIFT Lab.

\bibliographystyle{abbrv}
\bibliography{unique}

\end{document}